\documentclass[a4aper]{amsart}
\usepackage{cfr-lm}
\usepackage[osf]{libertine}

\usepackage{amssymb,amsmath,amsthm,amsfonts, mathrsfs, mathtools}
\usepackage{hyperref}
\hypersetup{colorlinks,breaklinks,
            linkcolor=[rgb]{0,0,0},
            citecolor=[rgb]{0,0,0},
            urlcolor=[rgb]{0,0,0}}
\usepackage{enumitem}
\usepackage{xcolor}
\usepackage{subfig, pgfplots}
\pgfplotsset{compat=1.7}

\usepackage{todonotes}

\newcommand{\loc}{\mathrm{loc}}
\newcommand{\R}{\mathbb{R}}

\newcommand{\N}{\mathbb{N}}

\newcommand{\eps}{\varepsilon}

\newcommand{\caps}{\mathrm{cap}_s}
\DeclareMathOperator*{\gammalim}{\Gamma-lim}

\newcommand{\abs}[1]{\left| #1 \right|}
\newcommand{\norm}[2]{\left\| #1 \right\|_{#2}}
\newcommand{\ddfrac}[2]{\frac{ \displaystyle #1}{ \displaystyle #2}}
\newcommand{\dist}[2]{\mathrm{dist}(#1,#2)}
\renewcommand{\div}{\mathrm{div}}

\newtheorem{proposition}{Proposition}[section]
\newtheorem{theorem}[proposition]{Theorem}
\newtheorem*{theorem*}{Theorem}
\newtheorem{corollary}[proposition]{Corollary}
\newtheorem{lemma}[proposition]{Lemma}

\theoremstyle{definition}
\newtheorem{definition}[proposition]{Definition}
\newtheorem{remark}[proposition]{Remark}
\numberwithin{equation}{section}

\begin{document}
\title[Regularity results for segregated configurations]
{Regularity results for segregated configurations \\ involving fractional Laplacian}
\date{}

\author[G. Tortone]{Giorgio Tortone}\thanks{}
\address{Giorgio Tortone \newline \indent Dipartimento di Matematica
	\newline\indent
	Alma Mater Studiorum Universit\`a di Bologna
	\newline\indent
	 Piazza di Porta San Donato 5, 40126 Bologna, Italy}
\email{giorgio.tortone@unibo.it}


\author[A. Zilio]{Alessandro Zilio}\thanks{}
\address{Alessandro Zilio \newline \indent Universit\'{e} Paris Diderot, Universit\'e de Paris,\newline \indent  Laboratoire Jacques-Louis Lions (CNRS UMR 7598),\newline \indent  8 place Aur\'elie Nemours, 75205, Paris CEDEX 13, France}
\email{azilio@math.univ-paris-diderot.fr}
\keywords{Free-boundary problem, optimal regularity, nonlocal diffusion, monotonicity formulas, segregation problems, variational methods}

\thanks{Work partially supported by the ERC Advanced Grant 2013 n.\ 339958 Complex Patterns for Strongly Interacting Dynamical Systems - COMPAT, held by Susanna Terracini. This work was partially supported by the project ANR-18-CE40-0013 SHAPO financed by the French Agence Nationale de la Recherche (ANR)}
\maketitle

\begin{abstract}
We study the regularity of segregated profiles arising from competition - diffusion models, where the diffusion process is of nonlocal type and is driven by the fractional Laplacian of power $s \in (0,1)$. Among others, our results apply to the regularity of the densities of an optimal partition problem involving the eigenvalues of the fractional
Laplacian. More precisely, we show $C^{0,\alpha^*}$ regularity of the density, where the exponent $\alpha^*$ is explicit and is given by
\[
	\alpha^* = \begin{cases}
		s &\text{for $s \in (0,1/2]$}\\
		2s-1 &\text{for $s \in (1/2,1)$}.
	\end{cases}
\]
Under some additional assumptions, we then show that solutions are $C^{0,s}$. These results are optimal in the class of H\"older continuous functions. Thus, we find a complete correspondence with known results in case of the standard Laplacian.
\end{abstract}


\section{Introduction}

In free-boundary problems, the regularity of the densities is a important step in showing the regularity and the geometrical structure of the solutions. This is true, in particular, for multi-phase problems, where the optimal regularity of the densities around the free-boundary is often a fundamental tool in deriving a description of the interface-set that separates the phases. The aim of this paper is to prove regularity results for a class of segregation models, covering in particular the optimal results.

Segregation models are a rather recent topic in free-boundary problems. These models usually describe two or more densities (them being population distributions, chemical compounds, components of different Bose-Einstein condensates) that are subject to diffusion (Brownian motion or jump-like processes) and strong negative interaction (competition, annihilation or repulsion). While the diffusion process tends to spread the densities homogeneously all over the domain, the negative interaction tends to disfavor the superposition of more than one density at each point. When combined, these two adversary processes bring forward pattern formation. From the point of view of the mathematical literature, this topic consists of many different areas of research, from existence theory of solutions, to regularity of the densities and of the free-boundary that emerges in the case of complete separation of the densities. Restricting ourselves to the topic of regularity of the densities, and to models that are similar to the ones we will consider in this manuscript, we can cite the contributions by Conti, Terracini and Verzini \cite{MR2151234}, Caffarelli and Lin \cite{MR2393430}, Noris, Tavares, Terracini and Verzini \cite{MR2599456}, and Soave and Zilio \cite{SoaveZilio_ARMA} in the case of standard diffusion processes. More recently, there have been advances in the theory of models with nonlocal diffusion by Terracini, Verzini and Zilio \cite{tvz1, tvz2} and of nonlocal competition by Caffarelli, Patrizi and Quitalo \cite{CaffPatrQuit} and Soave, Tavares, Terracini and Zilio \cite{zbMATH06870295}.

Here we are chiefly interested in the regularity of the densities in a segregation model that involves the fractional Laplacian of power $s$, for any $s \in (0,1)$. This paper is prompted by the quasi-optimal results contained in \cite{tvz1, tvz2, aletesi, tesi}. We start by recalling them here, together with some notation.

\noindent\textbf{Notation.} Let $s\in (0,1)$, $a=1-2s\in (-1,1)$. We denote $B_r(X) \subset \R^{n+1}$ is the ball of radius $r > 0$ centered at $X \in \R^{n+1}$. For any set $\Omega \subset \R^{n+1}$ we let $\Omega^+ = \Omega \cap \{y > 0\}$, $\partial^+ \Omega = \partial \Omega \cap \{y>0\}$, $\partial^0 \Omega = \partial \Omega \cap \{y = 0\}$. In particular, we let $S^{n-1}_r = \partial^0 B^r$. We consider the space of function $H^{1,a}(B_1)$ (see \cite{FKS,Nekvinda}), defined as the closure of $C^{\infty}(\overline{B_1})$ with respect the norm
\[
	\norm{u}{H^{1,a}(B_1)}^2 = \int_{B_1}{\abs{y}^a u^2\mathrm{d}X}+ \int_{B_1}{\abs{y}^a \abs{\nabla u}^2\mathrm{d}X}.
\]
We will always denote with $L_a = \mbox{div}(\abs{y}^a \nabla)$ the divergence form operator associated to the Muckenhoupt $A_2$ weight $X = (x,y) \mapsto  \abs{y}^a$. Given $\mathbf{u} = (u_1, u_2, \dots, u_k) \in H^{1,a}(B^+_1;\R^k)$ and $\mathbf{f} = (f_1, f_2, \dots, f_k) \in C(\R^k; \R^k)$ we define
\[
  \mathbf{f}(\mathbf{u}) = (f_1(u_1), f_2(u_2), \dots, f_k(u_k)).
\]
Similarly, letting
\[
  F_i(s) = \int_0^sf_i(t) \mathrm{d}t
\]
for $i=1,\cdots,k$, we introduce the function $\mathbf{F}\colon \R^k\to \R^k$, such that
\[
  \mathbf{F}(\mathbf{u}) = (F_1(u_1), F_2(u_2), \dots, F_k(u_k)).
\]
Finally, we consider the positive regularity exponent
\[
\alpha^* =
\begin{cases}
  s & 0<s \leq \frac{1}{2}, \\
  2s-1 & \frac{1}{2} < s < 1.
\end{cases}
\]

We are ready to state the results of interest in \cite{tvz1, tvz2, aletesi, tesi}.

\begin{theorem*} Let $\beta>0$, $\mathbf{f}_{\beta} \in C(\R^k; \R^k)$ be a collection of continuous functions, which map bounded sets into bounded sets uniformly with respect to $\beta$. Let $(\mathbf{u}_\beta)_\beta \in H^{1,a}(B^+_1;\R^k)$ be a family of solutions $\mathbf{u}_\beta = (u_{1,\beta},\dots,u_{k,\beta})$ of the system
\begin{equation}\label{beta}\tag{$P_\beta$}
\begin{cases}
- L_a u_{i,\beta} =0 
&\mathrm{in}\,\,\  B^+_1\\
-\lim_{y\to 0} y^a \partial_y u_{i,\beta} = f_{i,\beta}(u_{i,\beta})-\beta u_{i,\beta} \sum_{j\neq i} a_{ij} u_{j,\beta}^2 & \mathrm{on}\,\,\  \partial^0 B^+_1.
\end{cases}
\end{equation}
Let us assume that
\[
  \norm{\mathbf{u}_\beta}{L^\infty (B^+_1)} \leq M
\]
for a constant $M>0$ which is independent of $\beta$. Then, for any $\alpha \in (0,\alpha^*)$
\[
   \norm{\mathbf{u}_{\beta}}{C^{0,\alpha}(\overline{B^+_{1/2}})} \leq C,
\]
where $C=C(M,\alpha)$ is independent of $\beta$. Moreover, $(\mathbf{u}_\beta)_\beta$ is relatively compact in $H^{1,a}(B^+_{1/2})\cap C^{0,\alpha}(\overline{B^+_{1/2}})$, for any $\alpha \in (0,\alpha^*)$. Any accumulation point  $\mathbf{u}_\infty$  of the family  $(\mathbf{u}_\beta)_\beta$ when $\beta \to +\infty$ verifies
\[
	u_{i,\infty} u_{j,\infty} |_{y=0} \equiv 0 \qquad \text{for any $i \neq j$}.
\]
\end{theorem*}

Thanks to the local realization of the fractional Laplacian as a Dirichlet-to-Neumann map \cite{MR2354493}, the previous result implies a global counterpart for a nonlocal problem, either set on the whole of $\R^n$ or in smooth domains with Dirichlet boundary.

\begin{theorem*}
Let $\beta>0$, $(f_{i,\beta}: \R \to \R)_\beta$ be a collection of continuous functions, which map bounded sets int bounded sets uniformly with respect to $\beta$. Let $(\mathbf{u}_\beta)_\beta \in H^{s}(\R^n;\R^k)$ be a family of solutions $\mathbf{u}_\beta = (u_{1,\beta},\dots,u_{h,\beta})$ of
\begin{equation*}
\begin{cases}
(-\Delta)^s u_{i,\beta} = f_{i,\beta}(u_{i,\beta})-\beta u_{i,\beta} \sum_{j\neq i} a_{ij} u_{j,\beta}^2 & \mathrm{in}\,\,\  \Omega\\
u_{i,\beta} \equiv 0 & \mathrm{in}\,\,\ \R^n\setminus\Omega,
\end{cases}
\end{equation*}
where $\Omega$ is either the whole $\R^n$ or a domain of $\R^n$ with uniformly smooth boundary. Let us assume that
\[
  \norm{\mathbf{u}_\beta}{L^\infty (\Omega)} \leq M
\]
for some constant $M>0$ independent on $\beta$. Then, for any $\alpha \in (0,\alpha^*)$
\[
  \norm{\mathbf{u}_{\beta}}{C^{0,\alpha}(\R^n)} \leq C,
\]
where $C=C(M,\alpha)$ is independent of $\beta$.  Moreover, $(\mathbf{u}_\beta)_\beta$ is relatively compact in $H^{s}_\mathrm{loc} \cap C^{0,\alpha}_\mathrm{loc}$, for any $\alpha \in (0,\alpha^*)$. Any accumulation point  $\mathbf{u}_\infty$  of the family $(\mathbf{u}_\beta)_\beta$ when $\beta \to +\infty$ verifies
\[
	u_{i,\infty} u_{j,\infty} \equiv 0 \qquad \text{for any $i \neq j$}.
\]
\end{theorem*}

As mentioned in the previous statements, an important consequence of these results is that they imply a very useful compactness criterion for the solutions when $\beta \to +\infty$. If, moreover, we assume that the nonlinearities $(f_{i,\beta})_\beta$ converge uniformly on compact sets to some smooth function $f_i$ as $\beta \to +\infty$ with $f_i(0) = 0$, then also system $(P_\beta)$ passes to the limit (see below) and we can prove that the limit profiles belong to the class of segregated configurations $\mathcal{G}^s$ introduced in \cite{tvz1, tvz2, aletesi}. We recall here its definition.

\begin{definition}\label{class}
  Let $\mathcal{G}^s(B^+_1)$ stand for the set of functions $\mathbf{u} \in H^{1,a}_{\loc}(\overline{B^+_1};\R^k )$ whose components verify
 \begin{itemize}
   \item[(1)] $\mathbf{u} \in H^{1,a}(K\cap B^+_1) \cap C^{0,\alpha}(\overline{K\cap B^+_1})$, for every compact set $K\subset B$ and every $\alpha \in (0,\alpha^*)$;
   \item[(2)] $u_i\cdot u_j\lvert_{y=0}\equiv 0$ for every $i\neq j$ and they satisfy
\begin{equation}\label{eq}
  \begin{cases}
  -\div(y^a \nabla u_i) =0 &\mbox{in } B^+_1\\
  u_i\left(\lim_{y \to 0} y^a \partial_y u_i + f_i(u_i)\right)=0 & \mbox{on } \partial^0 B^+_1
  \end{cases}
 \end{equation}
where $f_i\colon\R\to\R$ are non-negative $\mathcal{C}^{1,\tau}$ functions, for some $\tau > 0$, and such that $f_i(0)=0$;
   \item[(3)] for every $X_0=(x_0,0) \in \partial^0 B^+_1$ and $r\in(0,\mbox{dist}(X_0,\partial B))$, the following Poho\v{z}aev type identity holds
   \begin{multline}\label{pohoz}
   (1-a-n)\int_{B_r^+(X_0)}{\abs{y}^a\abs{\nabla \mathbf{u}}^2\mathrm{d}X} + r\int_{\partial B_r^+(X_0)}{\abs{y}^a\abs{\nabla \mathbf{u}}^2\mathrm{d}\sigma}+\\
   +2n \int_{\partial^0 B_r^+(X_0)}{\sum_{i=1}^k F_i(u_i)\mathrm{d}x} -2r \int_{S^{n-1}_r(X_0)}{\sum_{i=1}^k F_i(u_i)\mathrm{d}x} =  2r \int_{\partial^+ B_r^+(X_0)}{\abs{y}^a (\partial_r \mathbf{u})^2 \mathrm{d}\sigma}.
   \end{multline}
 \end{itemize}
\end{definition}

\begin{remark}
The identities in \eqref{pohoz} are reminiscent of the classical Poho\v{z}aev identity. We point out that these identities contain integrals on set of co-dimension 2. These are meaningful because the densities $\mathbf{u}$ are continuous by definition.
\end{remark}

\begin{remark}
Since the functions $f_i$ are assumed $\mathcal{C}^{1,\tau}$ for some $\tau > 0$, we have that there exists a constant $C>0$ such that for any $s,t \in \R$
\[
  \left|f_i(s)-f_i(t)-f_i'(t)(s-t) \right| \leq C (s-t)^{1+\tau}.
\]
We observe that, by assumption, $\mathbf{u} \in L^\infty$. Thus it is sufficient to assume the previous inequality holds locally, that is for any $s, t \in [-\|\mathbf{u}\|_{L^\infty}, \|\mathbf{u}\|_{L^\infty}]$.
\end{remark}

In the particular case $s = 1/2$, i.e.\ $a=0$, in \cite{tvz1} it is shown that solutions $\mathbf{u} \in \mathcal{G}^{1/2}(B_1^+)$ are H\"older continuous of exponent $1/2$. This further improvement of regularity, obtained in the limit of segregation, is known to be optimal, and it is crucial in the study of the regularity and geometric properties of the common nodal set of the solutions (see \cite{desilva.terracini}).

On the contrary, for the generic exponent $s \in (0,1)$, the optimal regularity of solutions in $\mathcal{G}^s(B^+_1)$ is not covered by the previous results.  The aim of this paper is precisely to fill this gap. Specifically, we show here that

\begin{theorem}\label{thm main}
Let $\mathbf{u} \in \mathcal{G}^s(B_1^+)$. Then $\mathbf{u} \in C^{0,\alpha^*}(K \cup B_1^+)$ for any compact set $K \subset B_1$.
\end{theorem}

This regularity result can be used to show an improvement of the regularity of limit profiles in the case non-local segregation models. For instance, we can easily show the following.
\begin{corollary}
Let $\beta>0$, $(f_{i,\beta}: \R \to \R)_\beta$ be a collection of continuous functions, which map bounded sets into bounded sets uniformly with respect to $\beta$. Assume, moreover, that $f_{i,\beta} \to f_i$ locally uniformly for $\beta \to \infty$, where
\[
	f_i \in C^{1,\tau}_\loc(\R) \quad \text{and} \quad f_i(0) = 0 \qquad \text{for $i = 1, \dots, k$ and some $\tau > 0$.}
\]
Let $(\mathbf{u}_\beta)_\beta \in H^{s}(\R^n;\R^k)$ be a family of solutions $\mathbf{u}_\beta = (u_{1,\beta},\dots,u_{h,\beta})$ of
\begin{equation*}
\begin{cases}
(-\Delta)^s u_{i,\beta} = f_{i,\beta}(u_{i,\beta})-\beta u_{i,\beta} \sum_{j\neq i} a_{ij} u_{j,\beta}^2 & \text{in}\,\,\  \Omega\\
u_{i,\beta} \equiv 0 & \text{in}\,\,\ \R^n\setminus\Omega,
\end{cases}
\end{equation*}
where $\Omega$ is either the whole $\R^n$ or a domain of $\R^n$ with uniformly smooth boundary. Let us assume that
\[
  \norm{\mathbf{u}_\beta}{L^\infty (\Omega)} \leq M
\]
for some constant $M>0$ independent on $\beta$. Then, for any $\alpha \in (0,\alpha^*)$
\[
  \norm{\mathbf{u}_{\beta}}{C^{0,\alpha}(\R^n)} \leq C,
\]
where $C=C(M,\alpha)$ is independent of $\beta$. Moreover, $(\mathbf{u}_\beta)_\beta$ is relatively compact in $H^{s}_\mathrm{loc} \cap C^{0,\alpha}_\mathrm{loc}$, for any $\alpha \in (0,\alpha^*)$. The limit set of $\{\mathbf{u}_\beta\}_\beta$ is a subset of $H^{s}_\loc \cap C^{0,\alpha^*}_\loc$, and any accumulation point $\mathbf{u}_\infty$ verifies
\[
	\begin{cases}
		(-\Delta)^s u_{i,\infty} = f_{i}(u_{i,\infty}) & \text{in}\,\,\  \mathrm{int} \left(\mathrm{supp} \, u_{i,\infty} \right) \\
		u_{i,\infty} u_{j,\infty} \equiv 0 &\text{for any $i \neq j$}.
	\end{cases}
\]
\end{corollary}

As before, the previous result implies a regularity result for the trace of functions in $\mathbf{u}$, that are solutions to a nonlocal problem. Under some additional minimality assumptions (i.e., lack of self-segregation, see Corollary \ref{cor reg limit profile}), we can push further the regularity and show that solutions are actually $C^{0,s}$ regular. We have dedicated the last section of this manuscript to the exposition of such result. There we will prove the following result. We recall that a set $\omega \subset \R^n$ is an $s$-quasi open if for any $\eps > 0$ there exists an open set $\omega_\eps \subset \R^n$ such that $\caps(\omega \triangle \omega_\eps) \leq \eps$. We refer to Section \ref{minimal.solution} for the definition of the $s$-capacity.

\begin{theorem}
Let $k \geq 2$ and $\Omega \subset \R^{n}$ bounded and smooth domain. We consider $k$ $s$-quasi open and disjoint subsets of $\Omega$, denoted by $(\omega_1, \dots, \omega_k)$, and the associated functional
\begin{equation*}
  I(\omega_1, \dots, \omega_k) = \sum_{i=1}^{k} \lambda_{1,s}(\omega_i)
\end{equation*}
where $\omega \mapsto \lambda_{1,s}(\omega)$ is the generalized principle eigenvalue of $\omega$ with Dirichlet boundary conditions,
\[
    \lambda_{1,s}(\omega) = \inf\left\{ C_{n,s} \int_{\R^n}\int_{\R^n} {\frac{\abs{u(x)-u(y)}^2}{\abs{x-y}^{n+2s}}\mathrm{d}x\mathrm{d}y} : \begin{array}{l} u \in H^s(\R^n) \text{ with } \\\caps( \{ u \neq 0\} \setminus \omega ) = 0\\ \text{and } \|u\|_{L^2} = 1 \end{array} \right\}.
\]
There exist $s$-quasi open and disjoint sets $(\omega_1, \dots, \omega_k)$ that minimize the functional $I$. Moreover any minimizer of $I$ is equivalent (in the $\caps$ sense) to $k$ disjoint open sets. The first eigenfunction corresponding to any optimal partition is a $C^{0,s}$ function and, together, their extension belongs to the class $\mathcal{G}_s$.
\end{theorem}


To conclude this presentation, we mention that recently in \cite[Section 9]{sire.terracini.tortone} the authors considered some degenerate operator related to the local realization of fractional power of divergence form operator with Lipschitz leading coefficient. The techniques that they have introduced are quite robust, as they are based on subtle variants of the Almgren monotonicity formula that are close to ours. For this reason, we believe that our main results hold true for a much wider class of non-local operator.

Before presenting the proofs of our results, we make some final remarks about this subject. In the case of the standard diffusion, i.e.\ $s=1$, the analysis of the free-boundary has been the subject of an extensive study. We cite \cite{MR2393430, MR2984134} and the references therein as main contributions. More recently, in \cite{desilva.terracini} the authors studied the structure and regularity of the free-boundary of segregated configurations belonging to the class $\mathcal{G}^{1/2}(B^+_1)$. Our main result can be of interest in extending their analysis to the case $s \in (0,1)$.

The threshold $s=1/2$ in the definition of the exponent $\alpha^*$  is due to the phenomenon of self-segregation, which has not been excluded yet in the case $s\in (1/2,1)$. It consists in the possibility that the same density $u_i$ is locally present on the two sides of the free-boundary
\[
  \Gamma(\mathbf{u}) = \{X \in \partial^0 B^+_1 \colon \mathbf{u}(X)=\mathbf{0}\}.
\]
A more detailed discussion on this topic can be found in \cite{tvz2}, where a connection is drawn between this phenomenon and the fractional capacity of Riezs type (see also the last section of this paper). Excluding such phenomenon would directly imply a $C^{0,s}$ regularity result for the densities. This possibility is explored in the last section of this paper.

\medskip

\noindent\textbf{Structure of the paper.} The paper is organized as follows. In Section \ref{section.Almgren} we show Almgren's type monotonicity formulas for segregated critical configurations in $\mathcal{G}^s(B^+_1)$. In particular, we show a collection of monotonicity formulas emphasizing the differences between the local regularity of solution on the free-boundary and away from the free-boundary.
Section \ref{section.reg.limit} is devoted to the proof of the main result. Our strategy is based on the validity of a Morrey type inequality for degenerate and singular operator. Finally, in Section \ref{minimal.solution} we apply the results obtained to study a problem of optimal partition involving the eigenvalues of the fractional Laplacian of order $s\in (0,1)$. In this case, we are able to exclude the phenomenon of self-segregation and improve the regularity result to the optimal exponent $\alpha = s$.


\section{Almgren's type monotonicity formulas}\label{section.Almgren}
Functions belonging to $\mathcal{G}^s(B^+_1)$ have a very rich structure. Mainly thanks to the validity of the Poho\v{z}aev identities, we are able to prove some monotonicity formulas of Almgren type. These formulas will be crucial in proving of the regularity result.

The solutions have different local behaviors, according to the location of the point around which we analyze the functions: free-boundary points, points inside the support of the trace of one of the functions, points inside the upper half-spaces.  In this section, our main goal is to show a series of monotonicity formulas of Almgren's type in order to analyze all possible cases.

\begin{remark}
Throughout this section we will consider that $\mathbf{u} \not \equiv \mathbf{0}$ in $B_1^+$. Observe that for $y > 0$, each component of $\mathbf{u}$ verifies an elliptic equation that has locally smooth and strictly positive diffusion coefficient. Hence it follows that $\mathbf{u}$ is locally smooth (actually $C^\infty$) in $B_1^+$ and, by the standard unique continuation property \cite{garofalo.lin}, we also find that $\mathbf{u} \neq \mathbf{0}$ a.e. in $B_1^+$.
\end{remark}


\subsection{Monotonicity formulas on the free-boundary}

In this first part, we consider the case of points $X_0 \in \partial^0 B^+_1$ on the common nodal set $\Gamma(\mathbf{u})$, that is such that $\mathbf{u}(X_0) = \mathbf{0}$. We follow the main ideas of the last section of \cite{tvz1}, which is concerned with the regularity of profiles in $\mathcal{G}^{1/2}$. With respect to the mentioned paper, here we find some technical issues, spawned by the weight $y^a$, that have to be dealt with.

Let then $X_0 \in \partial^0 B^+_1$ with $\mathbf{u}(X_0) = \mathbf{0}$ and $r > 0$ such that $B^+_r(X_0)\subset B^+_1$, we define
\begin{equation}\label{E and H}
  \begin{split}
    E(r) = E(X_0,\mathbf{u},r) &= \frac{1}{r^{n-1+a}}\left(\int_{B^+_r(X_0)}{y^a \abs{\nabla \mathbf{u}}^2 \mathrm{d}X} - \int_{\partial^0B^+_r}{\langle \mathbf{u}, \mathbf{f}(\mathbf{u})\rangle \mathrm{d}x } \right)\\
    H(r) = H(X_0,\mathbf{u},r) &= \frac{1}{r^{n+a}}\int_{\partial^+B^+_r(X_0)}{y^a \mathbf{u}^2 \mathrm{d}\sigma}
  \end{split}
\end{equation}
and, whenever $H(x_0,\mathbf{u},r)\neq 0$, Almgren's frequency formula as
\begin{equation}\label{eqn N}
  N(r) = N(x_0,\mathbf{u},r)=\frac{E(x_0,\mathbf{u},r)}{H(x_0,\mathbf{u},r)}.
\end{equation}
We aim at showing that the previous frequency formula is monotone increasing in $r$, up to an explicit corrective term. To do this, we first need to ensure that the function $r \mapsto N(x_0,\mathbf{u},r)$ is well-defined. Then, we will prove some estimates for its derivative with respect to $r$.

\begin{lemma}
The functions $r\mapsto E(r)$ and $r\mapsto H(r)$ are well defined and locally absolutely continuous for any $0 < r < \dist{X_0}{\partial B^+_1}$.
\end{lemma}
\begin{proof}
The functions $r\mapsto E(r)$ and $r\mapsto H(r)$ are well defined since $\mathbf{u}\in H^{1,a}(B_1 )$. The absolute continuity of $r\mapsto E(r)$ follows directly by Fubini's theorem and the trace inequality for $H^{1,a}(B_1 )$ functions \cite[Theorem 2.11]{Nekvinda}. On the other hand, by multiplying each equation in \eqref{eq} by $u_i$, integrating by parts in $B_r^+(X_0)$ and summing for $i = 1,\dots,k$, we find the identity
\begin{equation}\label{E.H}
  E(r)= \frac{1}{r^{n-1+a}}\int_{\partial^+B^+_r}{y^a \langle\mathbf{u},\partial_r \mathbf{u}\rangle \mathrm{d}\sigma} = \frac{r}{2}\frac{d}{dr}H(r).
\end{equation}
This implies the local absolute continuity of the function $r\mapsto H(r)$.
\end{proof}

By the previous result, we find that also $r\mapsto N(r)$ is well defined and locally absolutely continuous for any $r$ such that $H(r) > 0$. Later (see Proposition \ref{mono}) we show that this is the case for any $r$ sufficiently small. This will entail the absolute continuity of $N$ for any $r > 0$ small.

We now consider the problem of estimating the derivative of $N(r)$ with respect to $r$, in order to show its monotonicity. To do this, we will need to control the terms in its derivative. We start with a Poincar\'e type inequality.

\begin{lemma}\label{lem poin}
Let $u\in H^{1,a}(B^+)$ and $p \in [2,p^\star]$, where $p^\star = 2n/(n-2s) = 2n/(n+a-1)$ is Sobolev's exponent for the fractional Laplacian. There exists a constant $C=C(n, p,a)$ such that
\begin{equation}\label{poincare}
  \left(\frac{1}{r^{n}}\int_{\partial^0 B^+_r}{\abs{u}^p\mathrm{d}x}\right)^{\frac{2}{p}} \leq C\left[ \frac{1}{r^{n-1+a}}\int_{B^+_r}{y^a\abs{\nabla u}^2\mathrm{d}X} + \frac{1}{r^{n+a}}\int_{\partial^+ B^+_r}{y^a |u|^2\mathrm{d}\sigma}\right]
\end{equation}
and
\begin{equation}\label{poincare2}
   \frac{1}{r^{n+a}}\int_{\partial^+ B^+_r}{y^a |u|^2\mathrm{d}\sigma} \leq C\left[ \frac{1}{r^{n-1+a}}\int_{B^+_r}{y^a\abs{\nabla u}^2\mathrm{d}X} + \left(\frac{1}{r^{n}}\int_{\partial^0 B^+_r}{\abs{u}^p\mathrm{d}x}\right)^{\frac{2}{p}}\right]
\end{equation}
for every $0 < r < 1$.
\end{lemma}
\begin{proof}
This result is a direct consequence of the characterization of the class of trace of $H^{1,a}(B^+_r)$, with $r\in (0,1)$, and the Sobolev embedding in the context of fractional Sobolev-Slobodeckij spaces.

For the first inequality \eqref{poincare}, by \cite[Theorem 2.11]{Nekvinda}, the traces of $H^{1,a}(B^+)$ function of the set $\partial^0B_r^+$ coincides with the Sobolev-Slobodeckij space $H^s(\partial^0B^+_r)$. This is defined as the set of all functions $v : \partial^0B_r^+ \to \R$ with a finite norm
\[
  \norm{v}{H^s(\partial^0B^+_r)} := \left(\int_{\partial^0B^+_r}{\abs{v}^2\mathrm{d}x} + \frac{C(n,s)}{2} \int_{\partial^0B^+_r}\int_{\partial^0B^+_r}{\frac{\abs{v(x)-v(z)}^2}{\abs{x-z}^{n+2s}}\mathrm{d}x\mathrm{d}z} \right)^{1/2},
\]
where the term
\begin{equation}\label{gagliardo}
  \left[v\right]_{H^s(\partial^0B^+_r)}=\left(\frac{C(n,s)}{2}\int_{\partial^0B^+_r}\int_{\partial^0B^+_r}{\frac{\abs{v(x)-v(z)}^2}{\abs{x-z}^{n+2s}}\mathrm{d}x\mathrm{d}z} \right)^{1/2}
\end{equation}
is the Gagliardo seminorm of $v$ in $H^s(\partial^0B^+_r)$. Since $\partial^0 B^+_r$ is a Lipschitz domain with bounded boundary, the fractional Sobolev inequality states that
\[
  \norm{v}{L^p(\partial^0 B^+_r)} \leq C \norm{v}{H^s(\partial^0B^+_r)},
\]
for every $p \in [2,p^\star]$, where $p^\star = 2n/(n-2s) = 2n/(n+a-1)$.

For the second inequality \eqref{poincare2}, we can show the result following the same steps of the more classical case $a=0$ (see for instance \cite[Lemma 4.2]{MR2599456}).
\end{proof}

We can use the previous result to prove two useful estimates for the functions $E$ and $H$. We have
\begin{lemma}\label{lemma1}
   For any $p \in [2,p^\star]$, there exist constants $C>0$ and $\overline{r}>0$, such that for every $X_0\in \partial^0 B^+_1$ and $0 < r < \min(\overline{r}, \dist{X_0}{\partial B^+})$, we have
  \[
    \left[ \frac{1}{r^{n}}\int_{\partial^0 B^+_r(X_0)}{\abs{\mathbf{u}}^p\mathrm{d}x}\right]^{\frac{2}{p}} \leq C\left(E(r)+H(r)\right)
  \]
and
\[
	\frac{1}{r^{n-1+a}} \int_{B^+_r(X_0)}{y^a\abs{\nabla \mathbf{u}}^2\mathrm{d}X} + \frac{1}{r^{n+a}}\int_{\partial^+ B^+_r(X_0)}{y^a \mathbf{u}^2\mathrm{d}\sigma} \leq C\left(E(r)+H(r)\right).
\]
\end{lemma}
\begin{proof}
  We prove explicitly the first estimate, as the proof of the second one is already contained in it. Recall that $\mathbf{f}$ is locally Lipschitz continuous with $\mathbf{f}(0) = 0$ and $\mathbf{u} \in L^\infty(B^+)$. By Poincar\'e's inequality \eqref{poincare} with $p = 2$, we obtain
  \begin{align*}
  \abs{\frac{1}{r^{n-1+a}}\int_{\partial^0 B^+_r}{\langle\mathbf{u},\mathbf{f}(\mathbf{u})\rangle\mathrm{d}x}}& \leq \frac{1}{r^{n-1+a}}\int_{\partial^0 B^+_r} \left|\mathbf{u}\right| \left|\mathbf{f}(\mathbf{u}) - \mathbf{f}(\mathbf{0}) \right| \mathrm{d}x \leq \frac{C}{r^{n-1+a}}\int_{\partial^0 B^+_r}{\mathbf{u}^2\mathrm{d}x}\\
  &\leq C_2 r^{1-a}\left[\frac{1}{r^{n-1+a}}\int_{B^+_r(X_0)}{y^a\abs{\nabla \mathbf{u}}^2\mathrm{d}X} + \frac{1}{r^{n+a}}\int_{\partial^+ B^+_r(X_0)}{y^a \mathbf{u}^2\mathrm{d}X}\right]
  \end{align*}
  for some constant $C_2 > 0$ that depends on the Lipschitz constant of $\mathbf{f}$ in $[-\|\mathbf{u}\|_{L^\infty}, \|\mathbf{u}\|_{L^\infty}]$. Since by assumption we have $1-a > 0$, there exists $\bar r > 0$ such that $0 \leq C_2 r^{1-a} < 1/2$ for any $r \in [0,\bar r]$. Thus, thanks to the previous estimates, we find the inequality
  \begin{equation}\label{tipotav}
  E(X_0,\mathbf{u},r) + H(X_0,\mathbf{u},r) \geq \frac12 \left[\frac{1}{r^{n-1+a}}\int_{B^+_r(X_0)}{y^a\abs{\nabla \mathbf{u}}^2\mathrm{d}X} + \frac{1}{r^{n+a}}\int_{\partial^+ B^+_r(X_0)}{y^a \mathbf{u}^2\mathrm{d}\sigma}\right],
  \end{equation}
  and we can conclude the proof by Poincar\'e's inequality with $p \in [2,p^\star]$.
\end{proof}

We now introduce two auxiliary functions. We recall that $\tau > 0$ (for simplicity, we assume $\tau < 2/(n-1)$) is the exponent of regularity of the functions $f_i \in C^{1,\tau}$ (see Definition \ref{class}). Let
\[
  \psi(r) = \psi(X_0,\mathbf{u},r) =  r \left(\frac{1}{r^n}\int_{\partial^0 B^+_r(X_0)}{\abs{\mathbf{u}}^{2+\tau}\mathrm{d}X}\right)^{\frac{\tau}{2+\tau}}
\]
and
\[
  \Psi(r) = \Psi(X_0,\mathbf{u},r) = \int_0^r{t^{-a}\left(1 + \psi'(t) \right)\mathrm{d}t}.
\]

\begin{lemma}\label{lem estimates phi psi}
The functions $r\mapsto\psi(r)$ and $r\mapsto\Psi(r)$ are well defined and absolutely continuous for $r\in (0,\mathrm{dist}(X_0,\partial B^+_1))$. Moreover, there exists a constant $C = C(a, \|\mathbf{u}\|_{L^\infty}) > 0$, such that for any $X_0 \in \partial^0B^+$ and $0 < r < \dist{X_0}{\partial B^+}$ we have
\[
	0 \leq \psi(r) \leq C r \qquad \text{and} \qquad 0 \leq \Psi(r) \leq C r^{1-a}.
\]
\end{lemma}
\begin{proof}
The proof follows by rather straightforward computations. First we have
\[
  0 \leq \psi(r) = r \left(\frac{1}{r^n}\int_{\partial^0 B^+_r(X_0)}{\abs{\mathbf{u}}^{2+\tau}\mathrm{d}X}\right)^{\frac{\tau}{2+\tau}} \leq C r \|\mathbf{u}\|_{L^\infty}^{\tau}.
\]
We also point out that the derivative of $\psi$ is positive. Then, concerning $\Psi$, we find
\[\begin{split}
	0 \leq \Psi(r) &= \int_0^r t^{-a}(1+\psi'(t)) \mathrm{d}t = \frac{r^{1-a}}{1-a} + \left[t^{-a}\psi(t)\right]_0^r + a \int_{0}^{r} t^{-1-a} \psi(t) \mathrm{d}t \\
	&\leq \frac{r^{1-a}}{1-a} + C r^{1-a} + |a| \int_{0}^{r} C t^{-a} \mathrm{d}t \leq C r^{1-a}. \qedhere
\end{split}
\]
\end{proof}
We can use the auxiliary functions in combination with Poincar\'e's inequality in order to bound uniformly the integral terms on sets of co-dimension 2. We have
\begin{lemma}\label{lemma2}
  There exist constants $C>0$ and $\overline{r}>0$ such that
  \[
  \frac{1}{r^{n-1}}\int_{S^{n-1}_r(X_0)}{|\mathbf{u}|^{2+\tau}\mathrm{d}\sigma} \leq C\left(E(r)+H(r)\right) \psi'(r),
  \]
  for every $X_0 \in \partial^0 B^+$ and $0 < r < \min(\overline{r}, \dist{X_0}{\partial B^+})$.
\end{lemma}
\begin{proof}
A direct computation yields the identity
\[
  \psi'(r) = \frac{1}{r} \psi(r) \left( 1-n\frac{\tau}{2+\tau} + r \frac{\tau}{2+\tau}
  \frac{\displaystyle \int_{S^{n-1}_r}{\abs{\mathbf{u}}^{2+\tau}\mathrm{d}\sigma}}{\displaystyle \int_{\partial^0 B^+_r}{\abs{\mathbf{u}}^{2+\tau}\mathrm{d}\sigma}} \right).
\]
Since $0 < \tau < 2/(n-1)$, from Lemma \ref{lemma1} we deduce
\[
  \left(E(r)+H(r)\right) \psi'(r) \geq C \frac{1}{r^{n-1}}\int_{S^{n-1}_r}{\abs{\mathbf{u}}^{2+\tau}\mathrm{d}\sigma}.\qedhere
\]
\end{proof}

We are now ready to prove that Almgren's frequency quotient is monotone up to a correction term.

\begin{proposition}\label{mono}
There exist constants $C$ and $\overline{r}>0$ such that, for any $X_0 \in \Gamma(\mathbf{u})$ we have $H(r)>0$ and $N(r)>0$ for every $0 < r < \min(\overline{r}, \dist{X_0}{\partial B^+})$. Moreover, the map
\[
  r\mapsto e^{C\Psi(r)}\left(N(r) +1 \right)
\]
is monotone increasing. Moreover $H(r) > 0$ for all  $0 < r < \min(\overline{r}, \dist{X_0}{\partial B^+})$ and we have
\[
  \lim_{r\to 0^+} N(r) \geq \alpha^*.
\]
\end{proposition}
\begin{proof}
First, we show the monotonicity of the following modified Almgren frequency formula
\begin{equation}\label{new.almgren}
\widetilde{N}(r) = \frac{E(r)}{H(r)}+ 1 = N(r)+1
\end{equation}
in a suitable open interval $(r_1,r_2)$. Observe that Lemma \ref{lemma1} yields
\[
  E(r)+ H(r) \geq 0 \implies \widetilde{N}(r) = \frac{E(r)}{H(r)}+ 1\geq 0,
\]
whenever $H(r) \neq 0$ and $r>0$ is small enough. Since we are considering $\mathbf{u} \not \equiv \mathbf{0}$, by continuity of the function $r\mapsto H(r)$ we can consider a open interval $(r_1,r_2)$ where $H(r)$ does not vanish. Recalling that $\mathbf{u}\in L^\infty(B^+_1)$, and each components of $\mathbf{f}=(f_1,\dots,f_k)$ is locally Lipschitz continuous with $f_i(0) = 0$, there exists a positive constant $C>0$ such that
\[
\abs{\langle \mathbf{u},\mathbf{f}(\mathbf{u})\rangle } \leq C \mathbf{u}^2\quad\mbox{and}\quad \abs{\mathbf{F}(\mathbf{u})} \leq C \mathbf{u}^2,
\]
for every $i=1,\dots,k$. Now, taking into account the Poho\v{z}aev identity \eqref{pohoz}, if we differentiate the function $r\mapsto E(r)$ we obtain
\begin{align*}
\frac{d}{dr}E(r)=& -\frac{n-1+a}{r^{n+a}}\bigg(\int_{B^+_r}{y^a\abs{\nabla \mathbf{u}}^2 \mathrm{d}X} -\int_{\partial^0B^+_r}{\langle \mathbf{u}, \mathbf{f}(\mathbf{u})\rangle \mathrm{d}x} \bigg) \\ &+\frac{1}{r^{n-1+a}}\int_{\partial^+B^+_r}{y^a\abs{\nabla \mathbf{u}}^2 \mathrm{d}\sigma}- \frac{1}{r^{n-1+a}}\int_{S^{N-1}_r}{\langle \mathbf{u}, \mathbf{f}(\mathbf{u})\rangle \mathrm{d}\sigma}\\
=&\,\, \frac{2}{r^{n-1+a}}\int_{\partial^+B^+_r}{y^a \abs{\partial_r \mathbf{u}}^2\mathrm{d}\sigma} +R(r).
\end{align*}
In order to estimate the remainder we need to exploit the regularity of the functions $\mathbf{f}$. Since $\mathbf{f} \in C^{1,\tau}$, we have that there exists $C > 0$ such that
\[
  \left| 2F_i(s) - sf_i(s) \right| \leq C |s|^{1+\tau}
\]
for all $s \in [-\|\mathbf{u}\|_{L^\infty}, \|\mathbf{u}\|_{L^\infty}]$. Hence, we obtain
\begin{align*}
\abs{R(r)} \leq& \,\,\frac{n-1+a}{r^{n+a}}\int_{\partial^0B^+_r(X_0)}{\abs{\langle\mathbf{u},\mathbf{f}(\mathbf{u})\rangle} \mathrm{d}x} + \frac{2n}{r^{n+a}}\int_{\partial^0 B_r^+(X_0)}{\sum_{i=1}^k\abs{ F_i(u_i)}\mathrm{d}x}\\
& +\frac{1}{r^{n+a-1}} \int_{S^{n-1}_r(X_0)}{\sum_{i=1}^k \abs{ 2F_i(u_i) - u_if_i(u_i) }\mathrm{d}x}\\
\leq & \, C \left[\frac{1}{r^{n+a}}\int_{\partial^0B^+_r(X_0)}{\mathbf{u}^2\mathrm{d}x}+\frac{1}{r^{n+a-1}}\int_{S^{n-1}_r(X_0)}{ \mathbf{u}^{2+\tau}\mathrm{d}\sigma}\right]\\
\leq & \, C r^{-a}\left(E(r)+H(r)\right)\left(1+\psi'(r)\right)
\end{align*}
where in the last estimate we made use of Lemma \ref{lemma1} and Lemma \ref{lemma2}. Therefore, differentiating the modified Almgren quotient and using the Cauchy-Schwarz inequality on $\partial^+ B^+_r$, we obtain
\begin{align*}
  \frac{d}{dr}\widetilde{N}(r) =&\, \ddfrac{\ddfrac{d}{dr}E(r)+\ddfrac{d}{dr}H(r)}{E(r)+H(r)} - \ddfrac{\ddfrac{d}{dr}H(r)}{H(r)} \\
\geq&\, \frac{2 H(r)}{r^{2n+2a-1}}\left[\int_{\partial^+B^+_r}{y^a \abs{\partial_r \mathbf{u}}^2\mathrm{d}\sigma}\int_{\partial^+B^+_r}{y^a\mathbf{u}^2\mathrm{d}\sigma} - \left(\int_{\partial^+ B^+_r}{y^a \langle\mathbf{u},\partial_r \mathbf{u}\rangle\mathrm{d}\sigma}\right)^2\right]\\
&\,-C \widetilde{N}(r) r^{-a}\left(1+ \psi'(r)\right)\\
\geq&\,-C \widetilde{N}(r) r^{-a}\left(1+ \psi'(r)\right).
\end{align*}
As a result, we find that the function
\begin{equation}\label{eqn mono tilde N}
  r\mapsto e^{C\Psi(r)}\widetilde{N}(r) = e^{C\Psi(r)}( N(r) + 1)
\end{equation}
is absolutely continuous and increasing for $r \in (r_1,r_2)$.

We now show that the function $H(r)$ is always strictly positive in the interval $(0,r_2)$, thanks to the monotonicity of the modified Almgren quotient. We start by taking the derivative of the logarithm of $r\mapsto H(r)$ in the open interval $(r_1,r_2)$. From \eqref{E.H}, we find that, for $r\in (r_1,r_2)$,
\begin{equation}\label{int}
\frac{d}{dr}\log H(r)=\frac{2}{r}N(r).
\end{equation}
By the monotonicity of the modified Almgren quotient, we have that
\[
  N(r) \leq e^{C\left[\Psi( r_2)-\Psi(r)\right]}\left(N(r_2)+1\right)-1 \leq e^{C\Psi(r_2)}\left(N(r_2)+1\right)-1 =: M
\]
where $M > 0$. Substituting this estimate in \eqref{int} and integrating the resulting inequality in $r$, we obtain
\[
  \frac{H(r_2)}{H(r)}\leq \left(\frac{r_2}{r}\right)^{2M} \qquad \text{that is} \qquad H(r) \geq H(r_2) \left(\frac{r}{r_2}\right)^{2M} > 0
\]
which implies that $H(r) > 0$ for any $r \in (0,r_2)$. As a result, the modified Almgren quotient is defined for all $r \in (0,r_2)$, and it can be extended for $r = 0$ by taking its limit for $r \to 0^+$.

Next we prove that the function in \eqref{eqn mono tilde N} has a positive strict minimum. More precisely, we show that
\begin{equation}\label{eqn lower bound N}
   e^{C\Psi(r)}( N(r) + 1) \geq e^{C\Psi(0)}( N(0) + 1) \geq \alpha^* +1
\end{equation}
for any $r \in (0,r_2)$. We reason by virtue of a contradiction. Assume that there exists $0 < \eps \leq \alpha^* +1$ such that
\[
  e^{C\Psi(0)}( N(0) + 1) = \alpha^* +1 - \eps.
\]
We recall that $\Psi(0) = 0$ and that $r \mapsto \Psi(r)$ is a non-negative and continuous function. Thus, by monotonicity of the modified Almgren quotient, we find that there exists $\hat r > 0$ such that
\[
  N(r) \leq \alpha^* - \frac{\eps}{2} \qquad \text{for all $r \in [0,\hat r]$}.
\]
We can go back to the identity \eqref{int} and integrate it over $(r, \hat r)$ to find
\[
  \frac{H(\hat r)}{H(r)} \leq \left(\frac{\hat r }{r}\right)^{2\alpha^* - \eps} \qquad \text{for all $r \in [0,\hat r]$}.
\]
Now, since by assumption $\mathbf{u}\in C^{0,\alpha}_\loc (B^+_1)$ for every $\alpha\in (0,\alpha^*)$ and $\mathbf{u}(X_0)=\mathbf{0}$, we find that
\[
  H(r) = \frac{1}{r^{n+a}}\int_{\partial^+B^+_r(X_0)}{y^a |\mathbf{u}|^2 \mathrm{d}\sigma} = \frac{1}{r^{n+a}}\int_{\partial^+B^+_r(X_0)}{y^a |\mathbf{u}-\mathbf{0}|^2 \mathrm{d}\sigma} \leq C_\alpha r^{2\alpha}
\]
for any $0 < r < \dist{X_0}{\partial B^+}$. Combining the two estimates, we obtain
\[
  H(\hat r) \hat r^{\eps - 2\alpha^*} \cdot r^{2\alpha^*-\eps} \leq H(r) \leq C_\alpha r^{2\alpha},
\]
for every $\alpha \in (0,\alpha^*)$. Hence, the contradiction follows choosing for $r$ sufficiently small.

Finally, we show that the threshold $r_2 = \min(\overline{r}, \dist{X_0}{\partial B^+})$ where $\bar r$ can be chosen independently of $X_0$. We consider \eqref{eqn lower bound N}, which we rewrite as
\[
  N(r) \geq (\alpha^* + 1)e^{-C\Psi(r)} - 1.
\]
Let $\eps>0$ be a small fixed constant. By Lemma \ref{lem estimates phi psi} we find that there exists $\bar r >0$ that depends only on $\eps$ and $\|\mathbf{u}\|_{L^\infty}$ such that
\[
  (\alpha^* + 1)e^{-C\Psi(r)} - 1 \geq \eps \qquad \text{for all $r \in (0,\bar r)$}.
\]
Indeed, it suffices to take $\bar r$ smaller than the radius in Lemma \ref{lemma1} and
\[
  \rho := \left[C\log \left(\frac{1+\alpha^*}{1+\eps}\right)\right]^{\frac{1}{1-a}}
\]
for some constant $C = C(\|\mathbf{u}\|_{L^\infty})$. Thus we find
\[
  N(r) \geq \eps \qquad \text{for all $0 < r < \min(\overline{r}, \dist{X_0}{\partial B^+})$}.
\]
Plugging this estimate in \eqref{int} and integrating in $(r,R)$, we find
\[
  H(R) \geq H(r) \left(\frac{R}{r}\right)^{2\eps}
\]
which implies that $H(R) > 0$ for all $0 < R < \min(\overline{r}, \dist{X_0}{\partial B^+})$.
\end{proof}

We conclude by showing an upper bound for the suitable local energy of the solutions.
\begin{corollary}\label{upper E alpha u}
Under the same assumptions of Proposition \ref{mono}, there exist constants $C > 0$ and $\bar r>0$  such that
\[
	\frac{1}{r^{n-1+a+2\alpha^*}} \int_{B_r^+(X_0)}{y^a \abs{\nabla \mathbf{u}}^2 \mathrm{d}X} + \frac{1}{r^{n+a+2\alpha^*}} \int_{\partial^+ B_r^+(X_0)}{y^a \abs{\mathbf{u}}^2 \mathrm{d}\sigma} \leq C \frac{E(R) + H(R)}{R^{2\alpha^*}}
\]
for all $X_0 \in \Gamma(\mathbf{u})$ and $0 < r < R = \min(\bar r, \dist{X_0}{\partial^+ B^+})$.
\end{corollary}
\begin{proof}
We have all the ingredients necessary for the proof.
Let $R = \min(\bar r, \dist{X_0}{\partial^+ B^+})$, by monotonicity of the modified Almgren quotient, we find
\begin{equation}\label{eqn bounds on N w}
	\alpha^* + 1 \leq e^{C\Psi(r)} (N(r) + 1) \leq e^{C\Psi(R)} (N(R) + 1)
\end{equation}
for all $0 < r < R$.  Solving the previous equation in $N(r)$, we obtain the following lower bound for the original Almgren quotient
\[
	N(r) \geq (\alpha^* + 1)e^{-C\Psi(r)} - 1.
\]
Thus, taking the derivate of the logarithm of $H$ we have
\[
	\frac{d}{dr} \log H(r) = \frac{2}{r} N(r) \geq \frac{2}{r} \left[\alpha^* +  (\alpha^* + 1)\left(e^{-C\Psi(r)}-1\right)\right]
\]
for all $0 < r < R$. Integrating in $[r,R]$ and using the estimate in Lemma \ref{lem estimates phi psi}, we find
\[
	\frac{H(r)}{r^{2\alpha^*}} \leq \frac{H(R)}{R^{2\alpha^*}} \mathrm{exp}\left(2(\alpha^*+1) \int_r^R{\frac{e^{C \rho^{1-a}}-1}{\rho}\mathrm{d}\rho}\right).
\]
We now multiply the previous estimate with the last inequality in \eqref{eqn bounds on N w}. This gives
\[
  \begin{split}
	\frac{E(r) + H(r)}{r^{2\alpha^*}} &\leq \frac{E(R) + H(R)}{R^{2\alpha^*}} \mathrm{exp}\left(C (\Psi(R) - \Psi(r)) + 2(\alpha^*+1) \int_r^R{\frac{e^{C \rho^{1-a}}-1}{\rho}\mathrm{d}\rho}\right) \\
	&\leq C' \frac{E(R) + H(R)}{R^{2\alpha^*}}
  \end{split}
\]
where we have introduced the constant
\[
  C' = \mathrm{exp}\left(C \Psi(1) + 2(\alpha^*+1) \int_0^1{\frac{e^{C \rho^{1-a}}-1}{\rho}\mathrm{d}\rho}\right)
\]
which is positive and finite since the function in the integral is positive and bounded. We observe that $C'$ does not depends on $R$ nor on $r$.

To conclude, we can apply Lemma \ref{lemma1}, in order to obtain a lower bound for the term $E(r) + H(r)$. Finally, we find that
\[
	\frac{1}{r^{n-1+a+2\alpha^*}} \int_{B_r^+(X_0)}{y^a \abs{\nabla \mathbf{u}}^2 \mathrm{d}X} \leq C \frac{E(R) + H(R)}{R^{2\alpha^*}}. \qedhere
\]
\end{proof}
Under a stronger assumption on the Almgren quotient, we can show a better control of the energy of the solutions. We have
\begin{corollary}\label{upper E 2s u}
Under the same assumptions of Proposition \ref{mono}, we assume moreover that
\[
  \inf_{X_0 \in \Gamma(\mathbf{u})\cap K} N(X_0,\mathbf{u},0^+) \geq s.
\]
Then, there exist constants $C > 0$ and $\bar r>0$  such that
\[
	\frac{1}{r^{n}} \int_{B_r^+(X_0)}{y^a \abs{\nabla \mathbf{u}}^2 \mathrm{d}X} +\frac{1}{r^{n+1}} \int_{\partial^+ B_r^+(X_0)}{y^a \abs{\mathbf{u}}^2 \mathrm{d}\sigma} \leq C \frac{E(R) + H(R)}{R^{2s}}
\]
for all $X_0 \in \Gamma(\mathbf{u})$ and $0 < r < R = \min(\bar r, \dist{X_0}{\partial^+ B^+})$.
\end{corollary}

\subsection{Monotonicity formulas away from the free-boundary, with \texorpdfstring{$y > 0$ }{y>0}} We now consider the case of points $X_0$ that are outside of the free-boundary $\Gamma(\mathbf{u})$. Our goal is to develop monotonicity formulas also for these points. Recently, the first author of the paper, working in collaboration with Y.\ Sire and S.\ Terracini, has developed in \cite{sire.terracini.tortone} a complete theory of the stratification properties for the nodal set of solutions of the equation
\[
    L_a u = \div(|y|^a \nabla u ) = 0 \qquad \text{in $\R^{n+1}$.}
\]
Their strategy was based on the introduction of monotonicity formulas that are similar to the ones we shall encounter in this section. For this reason, we will now state the results that we need and point to the specific statements in \cite{sire.terracini.tortone} that contain their proofs.

We start by considering points $X_0 \in B^+$, that is, points detached from the set $\{y = 0\}$. As a corollary of \cite[Proposition 3.7]{sire.terracini.tortone} and \cite[Corollary 3.9]{sire.terracini.tortone} we get
\begin{lemma}\label{Larmonic}
Let $ u \in H^{1,a}(B^+)$ be a $L_a$-harmonic function, that is a solution of
\[
    -L_a u = 0.
\]
For any $X_0 \in B^+$ and $0 < r < \min(y_0/2, \dist{X_0}{\partial B^+})$, let
\begin{equation}\label{generalized.almgren}
  N(r) = N(X_0,u,r) = \ddfrac{\frac{1}{r^{n-1+a}} \int_{B_r(X_0)} y^a|\nabla u|^2\mathrm{d}X}{\frac{1}{r^{n+a}} \int_{\partial B_r(X_0)} y^a |u-u(X_0)|^2}.
\end{equation}
Then $r\mapsto e^{3|a| r / y_0} N(X_0,u,r)$ is monotone increasing. Moreover we have
\[
  \lim_{r \to 0^+} N(r) \geq 1.
\]
\end{lemma}

We are mostly interested in the following consequence of the previous result.
\begin{lemma}\label{techn.2}
Let $X_0=(x_0,y_0)\in B^+$ and $u \in H^{1,a}(B^+)$ be $L_a$-harmonic. There exists a constant $C = C(a) > 0$, independent of $u$ and $X_0$, such that
\begin{itemize}
  \item if $y_0 \geq 2 \, \dist{X_0}{\partial^+B^+}$, then for any $0 < r < R = \dist{X_0}{\partial^+B^+}$ we have
\begin{equation}\label{2.13.part1}
  \frac{1}{r^{n-a}} \int_{B_r(X_0)}{ y^a |\nabla u|^2 \mathrm{d}X} \leq C \frac{1}{R^{n-a}} \int_{B_{R}(X_0)}{ y^a |\nabla u|^2\mathrm{d}X}
\end{equation}
  \item if $y_0 < 2 \, \dist{X_0}{\partial^+B^+}$, then for any $0 < r < R = \frac{y_0}{2}$ we have
\begin{equation}\label{2.13.part2}
  \frac{1}{r^{n-a}} \int_{B_r(X_0)}{ y^a |\nabla u|^2 \mathrm{d}X} \leq C \frac{1}{R^{n-a}} \int_{B_{R}(\bar{X}_0)}{ y^a |\nabla u|^2\mathrm{d}X}
\end{equation}
where $\bar{X}_0=(x_0,0)$ is the projection of $X_0$ onto $\{y=0\}$.
\end{itemize}
\end{lemma}

To prove the previous statement, we need some intermediate steps. First we observe that, without loss of generality, we can assume $u(X_0) = 0$. Indeed, it suffices to substitute the function $u$ with $u-u(X_0)$. Under this notation and convention, we introduce the functional
\[
  H(r) = H(X_0,u,r) = \frac{1}{r^{n}}\int_{\partial B_r(X_0)}{y^a |u|^2 \mathrm{d}\sigma}.
\]
We point out the different scaling exponent in the radius $r$ with respect to the one previously introduced (here we find as scaling factor $1/r^n$ instead of $1/r^{n+a}$ as in \eqref{E and H}). This is due to the fact that the operator $L_a$ is locally uniformly elliptic for $y > 0$. A direct computation (see also \cite[Proposition 3.7]{sire.terracini.tortone}) shows that
\[
  \frac{d}{dr}\log H(r) =\frac{2}{r}N(r) + \frac{a}{r}\frac{  \displaystyle\int_{\partial B_r(X_0)}{y^a \left(1- \frac{y_0}{y} \right) |u|^2\mathrm{d}\sigma}}{\displaystyle\int_{\partial B_r(X_0)}{y^a |u|^2\mathrm{d}\sigma}}.
\]
This identity is reminiscent of \eqref{int}, if not for the presence of a remainder term. Next, we estimate the remainder via a simple geometrical argument.

\begin{lemma}\label{lemma3}
For any $X_0=(x_0,y_0) \in B^+$, $0 < r < \min (y_0/2, \dist{X_0}{\partial^+B^+})$ and $u \in H^{1,a}$, we have
\[
  \abs{\frac{a}{r}\int_{\partial B_r(X_0)}{y^a\left(1-\frac{y_0}{y}\right) |u|^2\mathrm{d}\sigma}} \leq \frac{|a|}{y_0} \int_{\partial B_r(X_0)}{y^a |u|^2\mathrm{d}\sigma}.
\]
\end{lemma}
Now, let us introduce the auxiliary function
\[
  \phi(r) := \exp \left[ 2 \int_0^r \frac{1}{t} \left(1-e^{-\frac{3|a|}{y_0}t}\right) dt \right].
\]
We observe that $\phi$ is bounded in $[0,1]$, monotone increasing and such that $\phi(0) = 1$ and $C =\phi(y_0/2)$ is a constant that depends only on $a$.
\begin{lemma}\label{lem H_phi_exp}
Let $X_0=(x_0,y_0) \in B^+, 0 < r < \min (y_0/2, \dist{X_0}{\partial^+B^+})$ and $u$ be $L_a$-harmonic in $B^+$. Then, the function
  \[
    r \mapsto \frac{H(r)}{r^2} \phi(r) e^{\frac{|a|}{y_0}r}
  \]
is monotone increasing.
\end{lemma}
\begin{proof}
The monotonicity result follows immediately from Lemma \ref{lemma3}. Indeed we have
\[
  \frac{\phi(r)}{r^2} e^{\frac{|a|}{y_0}r} = \exp\left[ 2 \int_0^r \frac{1}{t} \left(1-e^{-\frac{3|a|}{y_0}t}\right) dt - 2\log r + \frac{|a|}{y_0}r\right],
\]
and Lemma \ref{Larmonic} yields
\begin{align*}
  \frac{d}{dr} \log  \frac{H(r)}{r^2} \phi(r)e^{\frac{|a|}{y_0}r} &= \frac{d}{dr} \log H(r) - \frac{2}{r}e^{-\frac{3|a|}{y_0}r} + \frac{|a|}{y_0}\\
   &= \frac{2}{r} N(r) +\frac{a}{r}\frac{  \displaystyle\int_{\partial B_r(X_0)}{y^a \left(1-\frac{y_0}{y}\right) |u|^2\mathrm{d}\sigma}}{\displaystyle\int_{\partial B_r(X_0)}{y^a |u|^2\mathrm{d}\sigma}} - \frac{2}{r}e^{-\frac{3|a|}{y_0}r} + \frac{|a|}{y_0} \\
   &\geq \frac{2}{r}e^{-\frac{3|a|}{y_0}r} \left[e^{\frac{3|a|}{y_0}r} N(r) - 1 \right] \geq 0. \qedhere
\end{align*}
\end{proof}
We are now in a position to prove Lemma \ref{techn.2}.

\begin{proof}[Proof of Lemma \ref{techn.2}]
First for $0 < r < r_2 = \min (y_0/2, \dist{X_0}{\partial^+B^+}) \leq 1$, combining Lemma \ref{Larmonic} and Lemma \ref{lem H_phi_exp}, we find
\begin{multline}\label{important}
    \frac{1}{r^{n+1}}\int_{B_r(X_0)}{y^a\abs{\nabla u}^2\mathrm{d}X} = N(r)\frac{H(r)}{r^2} \leq N(r_2) e^{\frac{4|a|}{y_0}(r_2-r)}\frac{\phi(r_2)}{\phi(r)}\frac{H(r_2)}{r_2^2}\\
    \leq e^{\frac{4|a|}{y_0}(r_2-r)}\frac{\phi(r_2)}{\phi(r)}\frac{1}{r_2^{n+1}}\int_{B_{r_2}(X_0)}{y^a\abs{\nabla u}^2\mathrm{d}X}
    \leq C e^{\frac{4|a|}{y_0} r_2} \frac{1}{r_2^{n+1}}\int_{B_{r_2}(X_0)}{y^a\abs{\nabla u}^2\mathrm{d}X}.
\end{multline}
This shows in particular the first alternative of the statement if $r_2 = \dist{X_0}{\partial^+B^+} \leq y_0/2$. Next, assuming that $r_2=y_0/2$, we have that
\[
 \frac{1}{r^{n+1}}\int_{B_r(X_0)}{y^a\abs{\nabla u}^2\mathrm{d}X} \leq C e^{2|a|}\frac{1}{(y_0/2)^{n+1}}\int_{B_{\frac{y_0}{2}}(X_0)}{y^a\abs{\nabla u}^2\mathrm{d}X}.
\]
To conclude, let $\bar{X}_0=(x_0,0)$ be the projection of $X_0$ onto $\{y=0\}$. We obtain
\begin{multline*}
  \frac{1}{r^{n-a}} \int_{B_r(X_0)} y^a |\nabla u|^2 \mathrm{d}X= r^{a+1} \frac{1}{r^{n+1}} \int_{B_r(X_0)} y^a |\nabla u|^2 \mathrm{d}X\\
  \leq C r^{a+1} \frac{1}{(y_0/2)^{n+1}} \int_{B_{\frac{y_0}{2}}(X_0)} y^a |\nabla u|^2 \mathrm{d}X\leq C \frac{1}{(y_0/2)^{n-a}} \int_{B_{\frac{y_0}{2}}(X_0)} y^a |\nabla u|^2 \mathrm{d}X\\
  \leq C \frac{(3 y_0 /2)^{n-a}}{(y_0/2)^{n-a}} \frac{1}{R^{n-a}} \int_{B_{R}(\bar{X}_0)} y^a |\nabla u|^2 \mathrm{d}X\leq C \frac{1}{R^{n-a}} \int_{B_{R}(\bar{X}_0)} y^a |\nabla u|^2 \mathrm{d}X
\end{multline*}
where we used the fact that $a+1> 0$ and $r \leq 1$.
\end{proof}

\subsection{Monotonicity formulas away from the free-boundary, with \texorpdfstring{$y = 0$ }{y=0}} We now consider the case of points of the set $\{y=0\}$ that are away from the common nodal set. We need to distinguish between two possibilities, according to the behavior of the trace of the function under consideration: either $u_i = 0$ on $\partial^0B^+$ or $u$ verifies a Neumann boundary condition $\partial^0B^+$.

We start with the former possibility.

\begin{lemma}[{\cite[Corollary 3.6]{sire.terracini.tortone}}]\label{almgren.odd}
Let $ u \in H^{1,a}(B^+)$ be a solution of
\[
  \begin{cases}
    -L_a u = 0 &\text{in $B^+$},\\
    u = 0 &\text{on $\partial^0 B^+$}.\\
  \end{cases}
\]
For any $X_0 \in \partial^0 B^+_1$, $r \in (0,\dist{X_0}{\partial^+ B^+})$, let
\[
  N(X_0,u,r) = \ddfrac{\frac{1}{r^{n-1+a}} \int_{B_r^+(X_0)} y^a|\nabla u|^2\mathrm{d}X}{\frac{1}{r^{n+a}} \int_{\partial^+ B_r^+(X_0)} y^a |u|^2\mathrm{d}\sigma}.
\]
Then $r\mapsto N(X_0,u,r)$ is monotone increasing and $\lim_{r \to 0^+} N(X_0,u,r) \geq 2s = 1-a$.
\end{lemma}

Once again, we are mainly interested in the following consequence of the previous result.
\begin{lemma}\label{techn.3}
Under the same assumptions of Lemma \ref{almgren.odd}, for any $0 < r < R \leq \dist{X_0}{\partial^+ B^+}$ we have
\[
  \frac{1}{r^{n+1-a}} \int_{B_r^+(X_0)}{ y^a |\nabla u|^2 \mathrm{d}X} \leq \frac{1}{R^{n+1-a}} \int_{B_{R}^+(X_0)}{ y^a |\nabla u|^2\mathrm{d}X}.
\]
\end{lemma}
\begin{proof}
The proof follows the same idea of the proof of Lemma \ref{techn.2}. Thus, we will only briefly sketch it. Let
\[
  H(r) = H(u, X_0, r) = \frac{1}{r^{n+a}} \int_{\partial^+ B_r^+(X_0)} y^a |u|^2\mathrm{d}\sigma
\]
and $E(r) = N(r) H(r)$. Exploiting the monotonicity of $N(r)$ we find, by direct computation, that
\[
  \frac{d}{dr} \log \frac{H(r)}{r^{4s}} = \frac{2}{r}N(r) - \frac{4s}{r} = \frac{2}{r} \left(N(r)-2s\right) \geq 0.
\]
Thus, for $0 < r < R \leq \dist{X_0}{\partial^+ B^+}$, we have
\[
  \frac{E(r)}{r^{4s}} = N(r) \frac{H(r)}{r^{4s}} \leq N(R) \frac{H(R)}{R^{4s}} = \frac{E(R)}{R^{4s}}.
\]
We conclude by substituting the expression of $E$ inside of the previous inequality.
\end{proof}

We now consider the case in which the function $u$ verifies a semi-linear boundary condition of Neumann type on $\partial^0 B^+$.

\begin{proposition}\label{prp mono E u}
Let $u\in H^{1,a}(B^+_1) \cap L^{\infty}(B^+_1)$ be a solution of
\begin{equation}\label{nonlinear.lip}
  \begin{cases}
    -L_a u = 0 &\text{in $B^+$}\\
    -\partial_y^a u = f(u) &\text{on $\partial^0 B^+$},
  \end{cases}
\end{equation}
with $f\in C^{1,\tau}$ for some $\tau > 0$. There exist constants $\bar r > 0$ and $C = C(n, a, \|f\|_{C^{1,\tau}}, \|w\|_{L^\infty}) >0$ such that
\[
	\frac{1}{r^{n}} \int_{B_r^+(X_0)}{y^a \abs{\nabla u}^2 \mathrm{d}X} \leq C \left[\frac{1}{R^n}\int_{B_R^+(X_0)}{y^a \abs{\nabla u}^2 \mathrm{d}X} + \frac{1}{R^{n+1}}\int_{\partial^+ B^+_R(X_0)}{y^a u^2 \mathrm{d}\sigma} + R^{2s} \right]
\]
for all $X_0 \in \partial^0 B^+$ and $0 < r < R = \min(\bar r, \dist{X_0}{\partial^+B^+})$.
\end{proposition}
\begin{proof}
For a fixed $X_0=(x_0,0)\in \partial^0 B^+$, we define the function $w\in H^{1,a}(B^+_1)$ as
\begin{equation}\label{def funct w}
  w(X) := u(X) - u(X_0) - \frac{1}{1-a}y^{1-a}f(u(X_0)).
\end{equation}
Observe that the function $y^{1-a}$ is an entire $L_a$-harmonic function in $\R^{n+1}_+$ with zero trace on $\{y=0\}$ and constant normal derivative. Thus, from \eqref{nonlinear.lip}, we find that $w$ solves
  \begin{equation}\label{subs}
  \begin{cases}
    -L_a w = 0 &\text{in $B^+$}\\
    -\partial_y^a w = V(x) w &\text{on $\partial^0 B^+$}
  \end{cases}
  \qquad \text{with} \; V(x)= \frac{f(u(x,0))-f(u(X_0))}{u(x,0)-u(X_0)} \in L^\infty(\partial^0 B^+).
\end{equation}
We now show a monotonicity formula of Almgren type for the function $w$. Later we will show how this implies the result of the original function $u$. Following standard computations (see also \cite[Proposition 9.11]{sire.terracini.tortone}), we introduce the functions
\begin{align*}
  E(r) = E(X_0,w,r)
  &=  \frac{1}{r^{n+a-1}}\left[\int_{B_r^+(X_0)}{y^a \abs{\nabla w}^2 \mathrm{d}X} - \int_{\partial^0 B^+_r(X_0)}{V w^2\mathrm{d}x}\right],\\
  H(r) = H(X_0,w,r) & = \frac{1}{r^{n+a}}\int_{\partial^+ B^+_r(X_0)}{y^a w^2 \mathrm{d}\sigma}
\end{align*}
and the associated Almgren quotient
\begin{equation}\label{Almgren.a}
  N(r) =N(X_0,w,r)= \ddfrac{E(X_0,w,r)}{H(X_0,w,r)}
\end{equation}
whenever the denominator $H(r) \neq 0$. We now follow the same strategy as Proposition \ref{mono}. For this reason, we omit some of the details. By Lemma \ref{lemma1}, assuming that $w \neq 0$, we find that there exists a radius $\bar r > 0$ such that if $0 < r < \min(\bar r, \dist{X_0}{\partial^+ B^+})$,
\begin{equation}\label{eqn E+H}
	E(r)+ H(r) \geq \frac12 \frac{1}{r^{n+a-1}} \int_{B_r^+(X_0)}{y^a \abs{\nabla w}^2 \mathrm{d}X}  \geq 0.
\end{equation}
Exploiting the continuity of the function $r \mapsto H(r)$, we can choose an open interval $(r_1,r_2)$, with $r_2 < \min(\bar r, \dist{X_0}{\partial^+B^+})$, such that $H(r) > 0$ and $r \in (r_1,r_2)$. By differentiating the functions $r\mapsto E(r)$ and $r \mapsto H(r)$ and using \eqref{subs}, we find
\begin{align*}
  \frac{d}{dr}E(r) & = \frac{2}{r^{n+a-1}}\int_{\partial^+ B^+_r}{\abs{y}^a (\partial_r w)^2 \mathrm{d}\sigma} + R(r),\\
  \frac{d}{dr}H(r) & = \frac{2}{r^{n+a}}\int_{\partial^+ B^+_r}{\abs{y}^a w\partial_r w \mathrm{d}\sigma},
\end{align*}
where the remainder term $R(r)$ is given by
\[
R(r) =
\frac{2}{r^{n+a}}\int_{\partial^0 B^+_r}{V w \langle x, \nabla w\rangle \mathrm{d}x} - \frac{1-n-a}{r^{n+a}}\int_{\partial^0 B^+_r}{V w^2\mathrm{d}x} - \frac{1}{r^{n-1+a}}\int_{S^{n-1}_r}{V w^2\mathrm{d}x}.
\]
In order to estimate the remainder, we rewrite the first integral in a way that it does not depend on the gradient of $w$. We let $F(s) = \int_0^s f(t) dt$. Since the function $f \in C^{1,\tau}$, we have that there exists a constant $C > 0$ such that
\[
	\left|F(u(X_0) + w)  - F(u(X_0)) - f(u(X_0))w - \frac12 f'(u(X_0))w^2 \right| \leq C |w|^{2+\tau}
\]
and
\[
  \left|V(x)w - f'(u(X_0)) w \right| = \left| f(u(X_0)+w)-f(u(X_0)) - f'(u(X_0)) w \right| \leq C |w|^{1+\tau}.
\]
For notation convenience, let
\[
  T(X) := \frac{F(u(X_0) + w)  - F(u(X_0)) - f(u(X_0))w}{w^2}.
\]
We find
\begin{multline*}
      \int_{\partial^0 B^+_r} V w \langle x , \nabla w \rangle \mathrm{d}x  = \int_{\partial^0 B_r^+}  \langle x , [w V(x)] \nabla w \rangle \mathrm{d}x \\
      = \int_{\partial^0 B_r^+} \langle x , \left[f(u) - f(u(X_0))\right] \nabla u \rangle \mathrm{d}x  = \int_{\partial^0 B_r^+} \langle x , \nabla \left[ F(u) - f(u(X_0))u \right]  \rangle \mathrm{d}x \\
      = \int_{\partial^0 B_r^+} \langle x , \nabla \left[ F(u(X_0)+w) - F'(u(X_0))w \right]  \rangle \mathrm{d}x \\
      = \int_{\partial^0 B_r^+} \langle x , \nabla \left[ T w^2 \right]  \rangle \mathrm{d}x  = r \int_{S^{n-1}_r}{T w^2\mathrm{d}\sigma} - n \int_{\partial^0 B_r^+} T w^2 \mathrm{d}x.
\end{multline*}
As a result, by Lemma \ref{lemma1} and Lemma \ref{lemma2}, we find there exists a numerical constant $C > 0$ and a radius $\bar r > 0$ such that for all $0 < r <\min(\bar r, \dist{X_0}{\partial^+B^+})$, the following estimate holds
\[
  \begin{split}
  \abs{R(r)} &\leq \frac{1}{r^{n+a}}\int_{\partial^0 B^+_r}{(n+a-1) \left|V -2n T \right| w^2 \mathrm{d}x} + \frac{1}{r^{n-1+a}}\int_{S^{n-1}_r}{\left|2T - V\right| w^2\mathrm{d}x} \\
  &\leq C \left(\frac{1}{r^{n+a}} \int_{\partial^0 B_r^+} w^{2} \mathrm{d}x +  \frac{1}{r^{n-1+a}} \int_{S^{n-1}_r}{w^{2+\tau}\mathrm{d}\sigma} \right) \\
  &\leq C r^{-a}(1+\psi'(r))( E(r)+H(r)).
  \end{split}
\]
Here the function $r\mapsto \psi(r)$ stands for the function in Lemma \ref{lem estimates phi psi}, suitably redefined. Therefore, differentiating the Almgren quotient and using the Cauchy-Schwarz inequality on $\partial^+ B^+_r$, we obtain
\begin{align*}
\frac{d}{dr}\widetilde{N}(r) =&\, \ddfrac{\ddfrac{d}{dr}E(r)+\ddfrac{d}{dr}H(r)}{E(r)+H(r)} - \ddfrac{\ddfrac{d}{dr}H(r)}{H(r)} \\
\geq & \frac{2 H(r)}{r^{2n+2a-1}}\left[\int_{\partial^+B^+_r}{\abs{y}^a (\partial_r w)^2\mathrm{d}\sigma}\int_{\partial^+B^+_r}{\abs{y}^a w^2\mathrm{d}\sigma} - \left(\int_{\partial^+ B^+_r}{\abs{y}^a \langle w,\partial_r w\rangle\mathrm{d}\sigma}\right)^2\right]\\
&-C r^{-a}(1+\psi'(r)) \tilde N(r)\\
\geq& -C r^{-a}(1+\psi'(r)) \tilde N(r).
\end{align*}
This, in turn, implies that the function
\[
  r\mapsto e^{C\Psi(r)}(N(r)+1)
\]
is increasing as far as $H(r)\neq 0$, where $\Psi(r)$ as in Lemma \ref{lemma1}. We can now follow closely the proof of Proposition \ref{mono} to show that $H(r) > 0$ for all $0 < r < \min(\bar r, \dist{X_0}{\partial^+ B^+})$. Thus $r \mapsto \tilde N(r)$ is defined for all $r>0$ small enough. Exploiting its monotonicity, we can also define
\[
  N(0) = \lim_{r\to 0} e^{C\Psi(r)}(N(r)+1) - 1.
\]
We now claim that $N(0) \geq s$. Actually, a stronger estimate holds, $N(0) \geq 1$. To prove it, we can replicate the analysis in \cite[Section 9]{sire.terracini.tortone} to show that $w \in C^{0,\alpha}(B^+)$, for every $\alpha \in (0,1)$. This gives the claim, as by the proof of Proposition \ref{mono}. Alternatively, we can show a weaker bound, that is in any case sufficient for our analysis. Indeed, by  \cite[Theorem 4.1]{alassane} we know that $w \in C^{\alpha}(B^+)$ for all $\alpha \in (0,2s)$, and in particular $w \in C^{0,s}(B^+)$ which gives $N(0)\geq s = (1-a)/2$. Anyway, following the same reasoning of Corollary \ref{upper E alpha u}, we find that
\[
  \begin{split}
	\frac{E(r) + H(r)}{r^{2s}} &\leq \frac{E(R) + H(R)}{R^{2s}} \mathrm{exp}\left(C (\Psi(R) - \Psi(r)) + 2(s+1) \int_r^R{\frac{e^{C \rho^{1-a}}-1}{\rho}\mathrm{d}\rho}\right) \\
	&\leq C \frac{E(R) + H(R)}{R^{2s}}
  \end{split}
\]
for a constant $C$ that is independent of $R$ nor $r$. From \eqref{eqn E+H} we infer that there exists yet another constant $C > 0$ such that
\[
  \frac{1}{r^{n}} \int_{B_r^+(X_0)}{y^a \abs{\nabla w}^2 \mathrm{d}X} \leq C \frac{E(R) + H(R)}{R^{2s}}
\]
for all $0 < r < R$. Thus, substituting the definition of $E$ and $H$, exploiting the boundedness of $w$ and Lemma \ref{lem poin}, we find
\begin{equation}\label{eqn mono E w}
	\frac{1}{r^{n}} \int_{B_r^+(X_0)}{y^a \abs{\nabla w}^2 \mathrm{d}X} \leq C \left[\frac{1}{R^n}\int_{B_R^+(X_0)}{y^a \abs{\nabla w}^2 \mathrm{d}X} + \frac{1}{R^{n+1}}\int_{\partial^+ B^+_R(X_0)}{y^a w^2 \mathrm{d}\sigma} \right],
\end{equation}
where $C = C(n, a, \|f\|_{C^{1,\tau}}, \|w\|_{L^\infty})$. From this monotonicity result we now derive the monotonicity formula for the original function $u$. It suffices to go back to the definition of the function $w$ in \eqref{def funct w} and solve in $u$. First we have that
\[
	\frac{1}{r^{n}} \int_{B_r^+(X_0)}{y^a \abs{\frac{1}{1-a}\nabla y^{1-a}}^2 \mathrm{d}X} = \frac{1}{r^{n}} \int_{B_r^+(X_0)} y^{-a} = C r^{1-a} = Cr^{2s}
\]
and
\[
	\frac{1}{r^{n+1}} \int_{\partial^+ B_r^+(X_0)}{y^a \abs{\frac{1}{1-a} y^{1-a}}^2 \mathrm{d}X} = \frac{1}{r^{n+1}} \int_{\partial^+ B_r^+(X_0)} y^{2-a} = C r^{2s}
\]
Thus, substituting the definition of $w$ in \eqref{eqn mono E w}, we find
\[
	\frac{1}{r^{n}} \int_{B_r^+(X_0)}{y^a \abs{\nabla u}^2 \mathrm{d}X} \leq C \left[\frac{1}{R^n}\int_{B_R^+(X_0)}{y^a \abs{\nabla u}^2 \mathrm{d}X} + \frac{1}{R^{n+1}}\int_{\partial^+ B^+_R(X_0)}{y^a u^2 \mathrm{d}\sigma} + R^{2s} \right]
\]
This concludes the proof of Proposition \ref{prp mono E u}.
\end{proof}
%
%

\section{Regularity of the limit profile}\label{section.reg.limit}
In this section we prove the main result on the regularity of the limit profile in $\mathcal{G}^s(B^+_1)$. The proof is based on a contradiction argument, involving, on one hand, a Morrey type inequality suited for the operator $L_a$ and, on the other hand, the energy estimates of the solutions deeply based on the validity of the Almgren monotonicity formulas of the previous section. We start by stating the result.
\begin{proposition}[Theorem \ref{thm main}]\label{prp reg limit profile}
Let $s \in (0,1)$ and $\mathbf{u}\in \mathcal{G}^s(B^+_1)$ be a limit profile. Then $\mathbf{u} \in C^{0,\alpha^*}_\loc(B^+_1)$, where
\[
\alpha^* =
\begin{cases}
  s, & 0<s \leq \frac{1}{2}, \\
  2s-1, & \frac{1}{2} < s < 1.
\end{cases}
\]
\end{proposition}
Moreover, under a stronger assumption (see Corollary \ref{upper E 2s u}), we can sharpen the result of Proposition \ref{prp reg limit profile}. This is done by emphasizing the role of the Almgren quotient on the free-boundary $\Gamma(\mathbf{u})$. We have
\begin{corollary}\label{cor reg limit profile}
Let $s \in (0,1)$ and $\mathbf{u}\in \mathcal{G}^s(B^+_1)$ be a limit profile. If for every compact $K\subset B_1$ we have
\begin{equation}\label{almgren.bound}
  \inf_{X_0 \in \Gamma(\mathbf{u})\cap K} N(X_0,\mathbf{u},0^+) \geq s,
\end{equation}
then $\mathbf{u} \in C^{0,s}_\loc(B_1^+)$.
\end{corollary}
Let us explain the assumption of Corollary \ref{cor reg limit profile}, which may seem arbitrary at first. The value of the Almgren quotient at a free-boundary point $X_0 \in \Gamma(\mathbf{u})$ is equal to the homogeneity degree of suitable blow-up limits of the function $\mathbf{u}$ around such point $X_0$, which, in turns, can be bounded from below by smallest growth at infinity of one-dimensional homogeneous function belonging to the $\mathcal{G}^s$ class (for a detailed derivation, see \cite[Section 7]{tvz1}, \cite[Section 3.4]{aletesi} and \cite[Section 1.6]{tesi}). In particular, exploiting the behavior of the fundamental solution of the fractional Laplacian in dimension one, it is possible to prove the following dichotomy:
\begin{itemize}
    \item either $N(X_0,\mathbf{u},0^+) = 2s-1$, in which case $s \in (1/2,1)$ and there exist $r > 0$ and $i \in \{1, \dots, k\}$ such that $\partial^0 B_r(X_0) \setminus \Gamma(\mathbf{u}) = \partial^0 B_r(X_0) \cap \{u_i > 0\}$ (a case we named \emph{self-segregation});
    \item or $N(X_0,\mathbf{u},0^+) \geq s$ and all the blow-up limit at $X_0$ contains at least two one non-zero components (thus the free-boundary separates at the limit, the support of the trace of at least two densities).
\end{itemize}
Hence, in the same spirit of \cite[Section 1.6]{tesi}, the absence of self-segregation equates to $N(X_0, \mathbf{u}, 0^+) \geq  s$ which, in turns, gives us that the densities are actually $C^{0,s}$ regular, the optimal regularity for this kind of problem.

We start by introducing a Morrey type inequality tailor made for the operator $L_a$.
\begin{lemma}\label{morrey}
Let $u \in H^{1,a}(B)$ and fix a compact set $K \subset B$. Assume that there exist constants $\lambda \in (0,1)$ and  $C > 0$ such that
\[
  \int_{B_r(X')} |y|^a |\nabla u|^2 \mathrm{d}X \leq C r^{2(\lambda-1)} \int_{B_r(X')} |y|^a \mathrm{d}X.
\]
for any $X' \in K$ and $0 < r < \dist{X'}{\partial B}$. Then $u \in C^{0,\lambda}(K)$.
\end{lemma}
\begin{proof}
We start by recalling a Poincar\'e-type inequality due to Fabes, Kenig and Serapioni \cite[Theorem 1.5]{FKS}: there exists $C = C(n,a) > 0$ such that, for any $u \in H^{1,a}(B_r(X'))$, $B_r(X') \subset B$, the following inequality holds
\[
  \int_{B_r(X')} |y|^a |u-\overline{u_{B_r}}|^2 \mathrm{d}X \leq C r^2 \int_{B_r(X')} |y|^a |\nabla u|^2 \mathrm{d}X
\]
where $\overline{u_{B_r}}$ is the average of $u$ in the ball $B_r(X')$, that is
\[
  \overline{u_{B_r}} = \frac{1}{|B_r(X')|} \int_{B_r(X')} u \mathrm{d}X.
\]
From the assumption we deduce that
\[
	\int_{B_r(X')} |y|^a |u-\overline{u_{B_r}}|^2 \mathrm{d}X \leq C r^{2\lambda} \int_{B_r(X')} |y|^a\mathrm{d}X.
\]
We now recall that the function $X \mapsto |y|^a$ is an $A_2$-weight: there exists a constant $C > 0$ such that
\[
  \frac{1}{|B|} \int_{B}|y|^a \mathrm{d}X \cdot \frac{1}{|B|} \int_{B}|y|^{-a} \mathrm{d}X \leq C \qquad \text{for any ball $B \subset \R^{n+1}$}.
\]
It follows that
\begin{multline*}
  \int_{B_r(X')}|y|^a  \left(\frac{1}{|B_r(X')|} \int_{B_r(X')} |u - \overline{u_{B_r}}|\mathrm{d}X \right)^2 \mathrm{d}X\\
  \leq  \int_{B_r(X')}|y|^a \left(\frac{1}{|B_r(X')|} \int_{B_r(X')} |u - \overline{u_{B_r}}| |y|^{\frac{a}{2}} |y|^{-\frac{a}{2}} \mathrm{d}X\right)^2 \mathrm{d}X \\
  \leq \left(\frac{1}{|B_r(X')|} \int_{B_r(X')}|y|^a \mathrm{d}X \cdot \frac{1}{|B_r(X')|} \int_{B_r(X')}|y|^{-a} \mathrm{d}X\right) \int_{B_r(X')} |y|^a |u - \overline{u_{B_r}}|^2 \mathrm{d}X \\\leq C \int_{B_r(X')} |y|^a |u - \overline{u_{B_r}}|^2\mathrm{d}X .
\end{multline*}
Combining the two previous estimates, we obtain that there exists a constant $C > 0$ such that
\[
  \frac{1}{|B_r(X')|} \int_{B_r(X')} |u - \overline{u_{B_r}}|  \mathrm{d}X\leq C r^{\lambda}
\]
that is
\[
  \sup_{X' \in K, r > 0}\frac{1}{r^{n + 1 + \lambda}} \int_{B_r(X') \cap K} |u - \overline{u_{B_r}}| \mathrm{d}X\leq C.
\]
Thus $u$ belongs to the Campanato space $\mathcal{L}^{1, n + 1 + \lambda}(K)$. By Campanato's embeddings \cite[Theorem 1.17]{troianiello}, we find that $u \in C^{0,\lambda}(K)$.
\end{proof}

\begin{corollary}\label{morrey2}
Let $a \in (-1,1)$. Let $X_0 \in \R^{n+1}$, $u \in H^{1,a}(B(X_0))$ and fix a compact set $K \subset B(X_0)$. Assume that there exist constants $\lambda \in (0,1)$ and $C > 0$ such that
\[
  \frac{1}{r^{n-1+2\lambda}} \int_{B_r(X)}  |y|^a |\nabla u|^2 \mathrm{d}X \leq C |y_{\max}|^a
\]
for any $X \in K$ and $0 < r < \dist{X}{\partial B(X_0)}$. Here $|y_{\max}| = \sup\{|y| : (x,y) \in B_r(X)\}$. Then $u \in C^{0,\lambda}(K)$.
\end{corollary}
\begin{proof}
It suffices to observe that there exist constants $C_1, C_2 > 0$ such that
\[
  C_1 r^{n+1} |y_{max}|^a \leq \int_{B_r(X)} |y|^a \mathrm{d}X\leq C_2 r^{n+1}|y_{max}|^a.
\]
To show these inequalities, we observe just that by invariance under scaling, translation in $x$ and reflection in $y$, the claim is equivalent to
\[
  C_1 |t+1|^a \leq \int_{t-1}^{t+1} |y|^a \mathrm{d}y\leq C_2 |t+1|^a
\]
for all $t \geq 0$. These inequalities are now immediate since the functions involved are continuous in $t$, strictly positive for $t \geq 0$ and of the same order when $t \to +\infty$. 
\end{proof}

\begin{proof}[Proof of Proposition \ref{prp reg limit profile}]
As anticipated at the beginning of this section, the proof of this result is based on a contradiction argument, involving the Morrey's type inequality of Corollary \ref{morrey2} and the energy estimates of Section \ref{section.Almgren}. Here we show that the solution $\mathbf{u} \in C^{0,\alpha^*}(\overline{B_{1/2}^+})$. Standard covering arguments allow to show that $\mathbf{u} \in C^{0,\alpha^*}(K)$ for any a compact $K\subset B$. Our strategy is the following: first of all, we prove that if a suitable Morrey quotient is unbounded, it must necessarily be unbounded when computed on point of the free-boundary $\Gamma(\mathbf{u})$. Then, we show that the quotient is bounded on $\Gamma(\mathbf{u})$, which will imply the global boundedness of the quotient, and thus the Proposition.

We introduce some notation. For any $X_0 \in \overline{B^+}$ and $r > 0$, we define the Morrey quotient
\begin{equation}\label{phi.morrey}
\Phi(X_0,r)= \frac{\abs{y_{\max}}^{-a}}{r^{n-1+2\alpha^*}}\int_{B_r(X_0)}{\abs{y}^a \abs{\nabla \mathbf{u}}^2 \mathrm{d}X},
\end{equation}
where $|y_{\max}| = \sup\{|y| : (x,y) \in B_r(X_0)\}$. Based on Corollary \ref{morrey2}, we assume by contradiction that there exists a sequence $X_n \in \overline{B_{1/2}^+}$ and $r_n \in (0,1/2)$ such that
\[
  \Phi(X_n,r_n) \to +\infty.
\]
We denote from now on $X_n = (x_n,y_n)\in \R^{n+1}$ with $x_n \in \R^n$ and $y_n \in \R$, and similarly $X'_n = (x_n', y_n')$.

We first show that if $\Phi(X_n,r_n) \to +\infty$, then $\Phi(X'_n,R_n) \to +\infty$ with $X'_n$ on $\partial^0 B_{1/2}^+$ and $R_n \to 0$. Indeed, since $\mathbf{u}\in H^{1,a}(B^+)$, it must be the case that $r_n\to 0$. Moreover, by Lemma \ref{techn.2}, we can always assume that $y_n =0$. Indeed, if we assume that $y_n\geq 2 \mathrm{dist}(X_n,\partial^+ B^+)$, by \eqref{2.13.part1} we get
\begin{align*}
\Phi(X_n,r_n) &= r_n^{1-a-2\alpha^*} \frac{\abs{y_n+r_n}^{-a} }{r^{n-a}_n}\int_{B_{r_n}(X_n)}{\abs{y}^a \abs{\nabla \mathbf{u}}^2 \mathrm{d}X} \\
&\leq C R_n^{1-a-2\alpha^*}  \frac{\abs{y_n+R_n}^{-a} }{R^{n-a}_n}\int_{B_{R_n}(X_n)}{\abs{y}^a \abs{\nabla \mathbf{u}}^2 \mathrm{d}X} = C\Phi(X_n,R_n),
\end{align*}
where $1-a\geq 2\alpha^*$ and $R_n= \mathrm{dist}(X_n,\partial^+ B^+) \geq 1/2$, and the right hand side is bounded uniformly since $\mathbf{u}\in H^{1,a}(B_1^+)$. As a result, we find that necessarily $y_n\leq 2 \mathrm{dist}(X_n,\partial^+ B^+)$. Thus, by \eqref{2.13.part2}, we obtain
\begin{align*}
\Phi(X_n,r_n) &= r_n^{1-a-2\alpha^*} \frac{\abs{y_n+r_n}^{-a} }{r^{n-a}_n}\int_{B_{r_n}(X_n)}{\abs{y}^a \abs{\nabla \mathbf{u}}^2 \mathrm{d}X} \\
&\leq C R_n^{1-a-2\alpha^*}  \frac{\abs{y_n+R_n}^{-a} }{R^{n-a}_n}\int_{B_{R_n}(\bar{X}_n)}{\abs{y}^a \abs{\nabla \mathbf{u}}^2 \mathrm{d}X}\\
&\leq C \abs{\frac{y_n}{R_n}+1}^{-a} \Phi(\bar{X}_n,R_n) \leq C \Phi(\bar{X}_n,R_n),
\end{align*}
where $R_n = y_n/2$. As a result, if $\Phi(X_n,r_n)$ is unbounded, so must be $\Phi(\bar{X}_n,R_n)$, with $\bar X_{n} = (\bar x_n, 0)$ and $R_n \to 0$. 

Next, we prove that if the Morrey quotient $\Phi(X_n,r_n) = \Phi((x_n, 0),r_n)$ is unbounded, it must be unbounded for a sequence of point $(X_n)_n \subset \Gamma(\mathbf{u}) \cap \partial^0 B_{1/2 + o_n(1)}$. Indeed, let us assume that $R_n = \mathrm{dist}(X_n,\Gamma(\mathbf{u}))>0$. Up to a relabelling, we have that
\[
 \begin{cases}
  -L_a u_1=0 &\mbox{in } B^+_{R_n}\\
  -\partial_y^a u_1 = f_1(u_1) & \mbox{on } \partial^0 B^+_{R_n}
  \end{cases} \qquad\text{while}\qquad
  \begin{cases}
  -L_a u_j =0 &\mbox{in } B^+_{R_n}\\
  u_j=0 & \mbox{on } \partial^0 B^+_{R_n}
  \end{cases} \quad
\text{for $j\neq 1$}.
\]
We now reason separately for the density $u_1$ and the densities $u_j$ for $j \neq 1$. For the density $u_1$, by Proposition \ref{prp mono E u}, we have for any $0 < r < {R_n}$,
\[
  \begin{split}
	\frac{1}{r^{n}} \int_{B_r^+(X_n)}{y^a \abs{\nabla u_1}^2 \mathrm{d}X}&\leq C \left[\frac{1}{R_n^n}\int_{B_{R_n}^+(X_n)}{y^a \abs{\nabla u_1}^2 \mathrm{d}X} + \frac{1}{{R_n^{n+1}}}\int_{\partial^+ B^+_{R_n}(X_n)}{y^a u_1^2 \mathrm{d}x} + {R_n}^{2s} \right]\\
	&\leq C \left[\frac{1}{{R_n^n}}\int_{B_{R_n}^+(X_n)}{y^a \abs{\nabla u_1}^2 \mathrm{d}X} + \frac{1}{{R_n^{n+1-a}}}\int_{\partial^0 B^+_{R_n}(X_n)}{u_1^2 \mathrm{d}x} + {R_n}^{2s} \right],
  \end{split}
\]
where for the second inequality we have used the Poincar\'e inequality \eqref{poincare2} in Lemma \ref{lem poin} in order to estimate the boundary contribution on $\partial^+ B^+_R$ with a contribution on $\partial^0 B^+_R$. We then consider the densities $u_j$. Lemma \ref{techn.3}, when applied to each component separately, yields
\[
  \frac{1}{r^{n+1-a}} \sum_{j\neq 1}\int_{B_r^+(X_n)}{ y^a |\nabla u_j|^2 \mathrm{d}X} \leq \frac{1}{R_n^{n+1-a}} \sum_{j\neq 1}\int_{B_{R_n}^+(X_n)}{ y^a |\nabla u_j|^2\mathrm{d}X}
\]
for any $0 < r < {R_n}$.
Thus, by summing the two inequalities and recalling the definition of the Morrey quotient \eqref{phi.morrey}, we obtain

\begin{align*}
\Phi(X_n,r_n) & = r^{1-a-2\alpha^*}\left(\frac{1}{r^n}\int_{B_r(X_0)}{\abs{y}^a \abs{\nabla u_1}^2 \mathrm{d}X}+ \frac{r^{1-a}}{r^{n+1-a}}\sum_{j\neq 1}\int_{B_r(X_0)}{\abs{y}^a \abs{\nabla u_j}^2 \mathrm{d}X}\right)\\
&\leq C\left(\Phi(X_n,R_n) + \frac{1}{R_n^{n+2\alpha^*}}\int_{\partial^0 B^+_R(X_n)}{\mathbf{u}^2\mathrm{d}x}+R_n^{4s-2\alpha^*}\right),
\end{align*}
with $R_n \leq \min\{\overline{r}, \dist{X_n}{\partial^+B^+},\mathrm{dist}(X_n, \Gamma(\mathbf{u}))\}$. Here we used the fact that $2\alpha^* \leq 1-a$, i.e. $\alpha^* \leq s$. Recalling that $\mathbf{u} \in H^{1,a}(B^+)$ and, in particular, $\mathbf{u} \in L^{2}(\partial^0 B^+)$, we can assume that $R_n = \mathrm{dist}(X_n, \Gamma(\mathbf{u})) \leq \min\{\overline{r}, \dist{X_n}{\partial^+B^+}\}$. Consequently, we find
\[
  \begin{split}
\Phi(X_n,r_n) &\leq C\left(\Phi(\bar{X}_n,2R_n) + \frac{1}{(2R_n)^{n+2\alpha^*}}\int_{\partial^0 B^+_{2R_n}(\bar{X}_n)}{\mathbf{u}^2\mathrm{d}x}+ R_n^{4s-2\alpha^*}\right)\\
  &\leq C\left(\Phi(\bar{X}_n,2R_n) + \frac{1}{(2R_n)^{n+a+2\alpha^*}}\int_{\partial^+ B^+_{2R_n}(\bar{X}_n)}{y^a \mathbf{u}^2\mathrm{d}x}+ 1\right)
  \end{split}
\]
where $(\overline{X}_n)_n \subset \Gamma(\mathbf{u}) \cap B_{1/2 + o_n(1)}$ is a sequence of points such that $\mathrm{dist}(X_n,\Gamma(\mathbf{u}))=\abs{X_n-\overline{X}_n}$. Here, in order to obtain the second inequality, we have used again the Poincar\'e inequality \eqref{poincare} in Lemma \ref{lem poin}. In this way, we are able to control the boundary contribution on $\partial^0 B^+_R$ with a contribution on $\partial^+ B^+_R$.

So far, we have shown that if the Morrey quotient is unbounded, then the functional
\[
	(X,r) \mapsto \frac{1}{r^{n-1+a+2\alpha^*}} \int_{B_r^+(X)}{y^a \abs{\nabla \mathbf{u}}^2 \mathrm{d}X} + \frac{1}{r^{n+a+2\alpha^*}}\int_{\partial^+ B^+_{r}(X)}{y^a \mathbf{u}^2\mathrm{d}x}
\]
is unbounded when computed on a sequence $(X_n,r_n)_n$ such that $X_n \in \Gamma(\mathbf{u}) \cap B_{1/2 + o_n(1)}$ and $r_n \to 0$.  We now conclude by showing that this functional actually uniformly bounded. To do this, we appeal to Corollary \ref{upper E alpha u}, which precisely states that there exist $C > 0$ and $\bar r > 0$ such that
\[
	\Phi(X, r) + \frac{1}{r^{n+a+2\alpha^*}}\int_{\partial^+ B^+_{r}(X)}{y^a \mathbf{u}^2\mathrm{d}x} \leq C \frac{E(X,\mathbf{u},R)+H(X,\mathbf{u},R)}{R^{2\alpha^*}},
\]
where $R = \min(\bar r, \dist{X}{\partial^+B^+})$.
\end{proof}
\begin{remark}
The proof of Corollary \ref{cor reg limit profile} coincides with
the previous one except for the last part where the estimate in \eqref{almgren.bound} allows to apply Corollary \ref{upper E 2s u} instead of Corollary \ref{upper E alpha u} and to reach the same contradiction.
\end{remark}


\section{Minimal solutions}\label{minimal.solution}

In this final section we apply the results obtained so far to study a problem about optimal partitions. Namely, we study the case of an optimal partition problem involving the eigenvalues of the fractional Laplacian of order $s \in (0,1)$. We show here that any minimizer of an optimal partition functional can be associated to a vector of functions in $\mathcal{G}_s$. Moreover, by exploiting the additional minimality condition that these configurations enjoy, we are able exclude the phenomenon of self-segregation and prove an optimal regularity result for the densities.

\begin{remark}
The results presented in this section can be extend to more general cost functionals (see \cite{ale}). The modifications are, for the most part, immediate. For this reason we have decided to consider only a special case that is of interest also in the applications. At the end of this section we will give an example of a much larger class of functionals to which the theory applies.
\end{remark}

We start by recalling some definitions (see for instance \cite{landkof, ritorto}). For a bounded set $A \subset \R^n$ we define its $s$-capacity $\caps(A)$ as
\[
  \caps(A) = \inf \left\{\|u\|^2_{H^s(\R^n)}: u \in C_0^\infty(\R^n), u|_{A} = 1 \right\}.
\]
We say that a subset $A$ of $\R^n$ is $s$-quasi open if there exists $O \subset \R^n$ open and such that $\caps(A \triangle O) = 0$. Here $A \triangle O$ is the symmetric difference of $A$ and $O$. Our interest in the notion of $s$-capacity stems from the following property: if $u \in H^s(\R^n)$ and $B \subset \R$ is an open subset of $\R$, then $\{x : u(x) \in B\}$ is $s$-quasi open.

Let $k \geq 2$ and $\Omega \subset \R^{n}$ bounded and smooth domain be fixed throughout this section. We consider the problem of finding $k$ $s$-quasi open and disjoint subsets of $\Omega$, $(\omega_1, \dots, \omega_k)$ that minimize the functional
\begin{equation}\label{eqn funct}
  I(\omega_1, \dots, \omega_k) = \sum_{i=1}^{k} \lambda_{1,s}(\omega_i).
\end{equation}
Here $\omega \mapsto \lambda_{1,s}(\omega)$ is the functional that associate to each subset of $\Omega$ its generalized principle eigenvalue, defined as
\[
  \lambda_{1,s}(\omega) = \inf\left\{[u]^2_{H^s} : u \in H^s(\R^n) \text{ with } \caps( \{ u \neq 0\} \setminus \omega ) = 0 \text{ and } \|u\|_{L^2} = 1 \right\}.
\]

In this section we aim at showing that optimal partitions exist, and that the eigenfunction associated to each partition has the highest possible regularity. In particular, they are all $C^{0,s}$ functions.

\begin{proposition}\label{prp existence minimal}
For any $k \in \N_0$, there exist $s$-quasi open and disjoint sets $(\omega_1, \dots, \omega_k)$ that minimize the functional \eqref{eqn funct}.
\end{proposition}

The study of the regularity of the minimizers of \eqref{eqn funct} is still an open problem. Here we show how this problem is actually related to the study of the nodal set of $\mathcal{G}_s$ functions.

Of course, one could study optimal partition problems by using a more direct approach. Indeed, it is possible to introduce a topology on the subsets of $\R^n$ that makes the functional \eqref{eqn funct} lower-semi-continuous and coercive, and from this we can deduce the existence of solutions. This is the approach adopted by A.\ Ritorto in \cite{ritorto} (see also \cite{ritorto2}), where the author proves existence results for a very large class of functional.

Here we use a different approach for mainly two reasons. First of all, Proposition \ref{prp existence minimal} follows as a simple corollary of the theory so far developed, once we adequately reformulate the original problem in terms of the eigenfunctions. Secondly, with this approach we can say more about the regularity of the minimal configuration, both in terms of the eigenfunctions and the geometry of the minimal sets \cite{desilva.terracini}.
We give an equivalent formulation of the problem of minimizing \eqref{eqn funct}. Consider functional $J : H^s_0(\Omega) \to \R$ defined as
\begin{equation}\label{eqn funct2}
  J(\mathbf{u}) = \begin{cases}
    \displaystyle \sum_{i=1}^k [u_i]^2_{H^s} &\text{if $\| u_i u_j\|_{L^1} = \delta_{ij}$ for all $i, j \in \{1, \dots, k\}$}, \\
    +\infty &\text{otherwise}.
  \end{cases}
\end{equation}

We state the equivalence of the two formulations is expressed in the following result. We omit the proof since it follows by the definition of the objects involved.
\begin{lemma}
  Let $(\omega_1, \dots, \omega_k)$ be $s$-quasi open and disjoint subsets of $\Omega$, and assume that $I(\omega_1, \dots, \omega_k) <\infty$. Then, letting $u_i \in H^s(\R^n)$ be the principal eigenfunction associated to $\omega_i$ for any $i = 1, \dots, k$ and $\mathbf{u} = (u_1, \dots, u_k)$, we have
  \[
    I(\omega_1, \dots, \omega_k) = J(\mathbf{u}).
  \]
  Conversely, let $\mathbf{u} = (u_1, \dots, u_k) \in H^s(\R^n)$ be vector of functions and assume that $J(\mathbf{u}) < +\infty$. Then letting $\omega_1 = \{u_i \neq 0\}$ for any $i = 1, \dots, k$, we have that $(\omega_1, \dots, \omega_k)$ are $s$-quasi open and disjoint subsets of $\Omega$ such that
  \[
    J(\mathbf{u}) = I(\omega_1, \dots, \omega_k).
  \]
\end{lemma}

We have
\begin{lemma}\label{lem sci-coe}
  The functional $J$ in \eqref{eqn funct2} is lower-semi-continuous and coercive on $H^s(\R^n)$ with respect to the weak $H^s$ topology. In particular, there exist functions $(u_1, \dots, u_k) \in H^s(\R^n)$ that minimize \eqref{eqn funct2}.
\end{lemma}
\begin{proof}
Since $H^s(\Omega)$ embeds compactly in $L^2(\Omega)$, the constraints are continuous in the weak topology of $H^s$. Moreover, when finite, the functional is the sum of the squares of the $H^s$ semi-norms of the components, which are lower-semi-continuous. Finally the coerciveness follows by the Poincar\'e inequality for $H^s$ functions.
\end{proof}

By Lemma \ref{lem sci-coe} and the direct method of calculus of variations, we deduce that $J$ admits at least a minimizer. Our goal here is to show that any minimizer of $J$ belongs to $\mathcal{G}_s$. Let then $\mathbf{\bar u} = (\bar u_1, \dots, \bar u_k)$ be any minimizer of $J$. We reason as in \cite{ale} and show the claimed property by means of an approximation procedure. Let $e(s) = \sqrt{1+s^2}$, we define
\begin{equation}\label{eqn functbeta}
  J_\beta(\mathbf{u}) = \begin{cases}
    \displaystyle \sum_{i=1}^k\left( [u_i]^2_{H^s} + \int_{\R^n} e(u_i-\bar u_i) \right)+ \beta \sum_{i < j} \int_{\R^n} u_i^2 u_j^2 &\text{if $\| u_i\|_{L^2} = 1$ for $i = 1, \dots, k$}, \\
    +\infty &\text{otherwise},
  \end{cases}
\end{equation}
and also the limit functional
\begin{equation}\label{eqn functinfty}
  J_\infty(\mathbf{u}) = \begin{cases}
    \displaystyle \sum_{i=1}^k\left( [u_i]^2_{H^s} + \int_{\R^n} e(u_i-\bar u_i) \right) &\text{if $\| u_i u_j\|_{L^1} = \delta_{ij}$ for all $i, j \in\{1, \dots, k\}$}, \\
    +\infty &\text{otherwise}.
  \end{cases}
\end{equation}
We have

\begin{lemma}\label{lem sci-coe2}
Let $H^s(\Omega)$ be endowed with its weak topology. The functionals $J_\beta$ in \eqref{eqn functbeta} and $J_\infty$ in \eqref{eqn functinfty} are (sequentially) lower-semi-continuous and coercive. They are also point-wise increasing in $\beta > 0$. As a result, as $\beta \to +\infty$, $J_\beta$ $\Gamma$-converges to $J_\infty$. Finally, any sequence of minimizers $\{\mathbf{u}_\beta\}$ of $J_\beta$ converges weakly in $H^s_0(\Omega )$ to $\mathbf{\bar u}$.
\end{lemma}
\begin{proof}
The result follows rather directly from the definition of $J_\beta$, $J_\infty$ and $J$. Let us show that $J_\beta$ is lower-semi-continuous and coercive. Let $\{\mathbf{u}_n\}_{n\in\N}$ be a sequence of functions in $H^s(\Omega )$ such that $\mathbf{u}_n \rightharpoonup \mathbf{u}$ in $H^s(\Omega )$. We need to show that $\liminf_{n\to\infty} J_\beta(\mathbf{u}_n) \geq J_\beta(\mathbf{u})$. We can assume, without loss of generality, that there exists $M > 0$ such that $J_\beta(\mathbf{u}_n) \leq M$ for all $n \in \N$. Then we have
\begin{itemize}
	\item the sequence $\{\mathbf{u}_n\}_k$ converges in $L^2(\Omega,\R^k)$. Indeed, $H^s(\Omega,\R^k)$ embeds compactly in $L^2(\Omega,\R^k)$. Thus $\|u_i\|_{L^2} = 1$ for $i = 1, \dots, k$ and $\mathbf{u}_n$ converges point-wise a.e.\ to $\mathbf{u}$;
	\item by Fatou's Lemma and point-wise a.e.\ convergence, we have for all $i=1,\dots,k$
	\[
		\int_{\Omega} e(u_{i}-\bar{u}_i) + \beta \sum_{i < j} u_{i}^2 u_j^2 \leq \liminf_{n\to\infty} \int_{\Omega}  e(u_{i,n}-\bar{u}_i) + \beta \sum_{i < j} u_{i,n}^2 u_{j,n}^2
	\]
	\item finally, since quadratic forms are lower-semi-continuous in the weak topology, we find that
	\[
		[u_{i}]_{H^s(\Omega)}^2 \leq \liminf_{n\to\infty} [u_{i,n}]_{H^s(\Omega)}^2 \qquad\text{for all $i=1,\dots,k$.}	
	\]
\end{itemize}
Thus we conclude that $J_\beta$ is lower-semi-continuous in the prescribed topology. From its definition we deduce that $J_\beta$ is also coercive, indeed
\[
	\{ \mathbf{u} \in H^s(\Omega ) : J(\mathbf{u}) \leq M \} \subset \left\{ \mathbf{u} : \sum_{i=1}^{k}[u_i]^2_{H^s} \leq M \text{ and } \|u_i\|_{L^2}= 1 \right\} \subset \overline{B}_{M+k}(\mathbf{0})
\]
and closed balls of $H^s_0(\Omega )$ are compact in the weak topology.

The family $J_\beta$ is evidently point-wise monotone increasing in $\beta$ and converges point-wise to $J$. By \cite[Proposition 5.4]{DalMaso} we deduce that
\[
	\gammalim_{\beta \to +\infty} J_\beta = J_\infty.
\]
Hence we find that also $J_\infty$ is lower-semi-continuous and coercive. We can reason in a similar way to show that $J$ is lower-semi-continuous and coercive.

To conclude, by \cite[Corollary 7.20]{DalMaso}, we have that any sequence of minimizers of $J_\beta$ converges to a minimizer of $J_\infty$. Since
\[
  J_\infty(\mathbf{u}) = J(\mathbf{u}) + \sum_{i=1}^k \int_{\Omega} e(u_{i,k}-\bar{u}_i)
\]
we find that, necessarily, $J_\infty$ has a unique minimizer, $\mathbf{\bar u}$
\end{proof}

Following now the arguments in \cite{tvz1,tvz2,ale}, we can show that the minimizer $\mathbf{\bar u} \in \mathcal{G}^s$. Actually, since the same reasoning hold for any vector $\mathbf{u}$ corresponding to a minimal partition, we find that they all belong to the class $\mathcal{G}^s$. As a result, by Theorem \ref{thm main} any optimal partition corresponds to a vector of $C^{0,\alpha^*}$ eigenfunctions. To conclude, we show that any minimizers of the optimal partitions are not afflicted by the phenomenon of self-segregation. This will finally gives us (Corollary \ref{cor reg limit profile}) that the densities are actually $C^{0,s}$ regular, the optimal regularity for this kind of problem.

\begin{lemma}
Let $x_0 \in\Omega$ and assume that there exists $r > 0$ small enough such that $B_r(x_0) \subset \Omega$,  $\{u_1 > 0\} \cap B_r(x_0) \neq \emptyset$ and $\{u_i > 0\} \cap B_r(x_0) = \emptyset$ and for all $i \geq 2$. Then $B_{r}(x_0) \subset \{u_1 > 0\}$.
\end{lemma}
\begin{proof}
It suffices to show that otherwise the corresponding partition $\omega_i = \{u_i > 0\}$ is not optimal. Observe that since $\mathbf{u} \in \mathcal{G}_s$ we can assume $\omega_i$ to be the largest open set equivalent to $\{u_i > 0 \}$. By assumption we have that
\[
  \omega_1 \subsetneqq \omega_1 \cup B_r(x_0)
\]
and
\[
  \left(\omega_1 \cup B_r(x_0) \right) \cap \omega_i = \emptyset \qquad \text{for all $i \geq 2$}.
\]
Let $\hat u_1$ be the first generalized eigenfunction of the set $\omega_1 \cup B_r(x_0)$ and let $\hat \omega_1 = \{\hat u_1 > 0\}$. By the strong maximum principle for the fractional Laplacian we know that the first eigenfunction $\hat u_1$ is strictly positive in $\omega_1 \cup B_r(x_0)$. We have that
\[
  \omega_1 \cup B_r(x_0) \subset \hat \omega_1 \subset \Omega \setminus \left( \bigcup_{i \geq 2} \omega_i\right).
\]
In particular $\{ \hat \omega_1, \omega_2, \dots, \omega_k\}$ is an admissible partition of $\Omega$. Thus, by monotonicity of the eigenvalue we find that
\[
  \lambda_1\left(\omega_1 \cup B_r(x_0) \right) = \lambda_1(\hat \omega_1) \leq  \lambda_1 \left(\omega_1 \right).
\]
If the inequality in the previous equation holds in a strict sense, then we can conclude that the original partition is not optimal. Let us assume by contradiction that
\[
  \lambda_1(\hat \omega_1 ) = \lambda_1 \left(\omega_1 \right).
\]
By definition it follows that $u_1$ is also the first eigenfunction of the set $\hat \omega_1$, as it shares the same Rayleigh quotient of the function $\hat u_1$ and it belongs to a smaller functional space. But the $u_1$ has to be strictly positive in $B_r(x_0) \subset \hat \omega_1$, a contradiction.
\end{proof}

\begin{remark}
We conclude by pointing out the that same strategy works for more general functionals. For instance, the same result holds for minimizers of
\[
  E(\mathbf{u}) = \begin{cases}
    \sum_{i=1}^k [u_i]^2_{H^s} + \int_{\Omega} m_i u_i^3 &\text{if $\| u_i u_j\|_{L^1} = \delta_{ij}$ for all $i, j \in\{1, \dots, k\}$}, \\
    +\infty &\text{otherwise}
    \end{cases}
\]
where $m_i \geq 0$. This functional appears in the study of Bose-Einstein condensates in the framework of  fractional quantum mechanics.
\end{remark}

\bibliographystyle{plain}
\bibliography{bibliography}

\begin{thebibliography}{10}

\bibitem{MR2393430}
L.~Caffarelli and F.-H. Lin.
\newblock Singularly perturbed elliptic systems and multi-valued harmonic
  functions with free boundaries.
\newblock {\em J. Amer. Math. Soc.}, 21(3):847--862, 2008.

\bibitem{CaffPatrQuit}
L.~Caffarelli, S.~Patrizi, and V.~Quitalo.
\newblock On a long range segregation model.
\newblock {\em J. Eur. Math. Soc. (JEMS)}, 19(12):3575--3628, 2017.

\bibitem{MR2354493}
L.~Caffarelli and L.~Silvestre.
\newblock An extension problem related to the fractional {L}aplacian.
\newblock {\em Comm. Partial Differential Equations}, 32(7-9):1245--1260, 2007.

\bibitem{MR2151234}
M.~Conti, S.~Terracini, and G.~Verzini.
\newblock A variational problem for the spatial segregation of
  reaction-diffusion systems.
\newblock {\em Indiana Univ. Math. J.}, 54(3):779--815, 2005.

\bibitem{DalMaso}
Gianni Dal~Maso.
\newblock {\em An introduction to {$\Gamma$}-convergence}, volume~8 of {\em
  Progress in Nonlinear Differential Equations and their Applications}.
\newblock Birkh\"{a}user Boston, Inc., Boston, MA, 1993.

\bibitem{desilva.terracini}
D.~{De Silva} and S.~{Terracini}.
\newblock {Segregated configurations involving the square root of the laplacian
  and their free boundaries}.
\newblock {\em ArXiv e-prints}, October 2018.

\bibitem{FKS}
E.~Fabes, C.~Kenig, and R.~Serapioni.
\newblock The local regularity of solutions of degenerate elliptic equations.
\newblock {\em Comm. Partial Differential Equations}, 7(1):77--116, 1982.

\bibitem{ritorto2}
Juli\'{a}n Fern\'{a}ndez~Bonder, Antonella Ritorto, and Ariel~Martin Salort.
\newblock A class of shape optimization problems for some nonlocal operators.
\newblock {\em Adv. Calc. Var.}, 11(4):373--386, 2018.

\bibitem{garofalo.lin}
Nicola Garofalo and Fang-Hua Lin.
\newblock Monotonicity properties of variational integrals, {$A_p$} weights and
  unique continuation.
\newblock {\em Indiana Univ. Math. J.}, 35(2):245--268, 1986.

\bibitem{landkof}
N.~S. Landkof.
\newblock {\em Foundations of modern potential theory}.
\newblock Springer-Verlag, New York-Heidelberg, 1972.
\newblock Translated from the Russian by A. P. Doohovskoy, Die Grundlehren der
  mathematischen Wissenschaften, Band 180.

\bibitem{Nekvinda}
A.~Nekvinda.
\newblock Characterization of traces of the weighted {S}obolev space
  {$W^{1,p}(\Omega,d^\epsilon_M)$} on {$M$}.
\newblock {\em Czechoslovak Math. J.}, 43(118)(4):695--711, 1993.

\bibitem{alassane}
Alassane Niang.
\newblock Boundary regularity for a degenerate elliptic equation with mixed
  boundary conditions.
\newblock {\em Communications on Pure \& Applied Analysis}, 18:107, 2019.

\bibitem{MR2599456}
B.~Noris, H.~Tavares, S.~Terracini, and G.~Verzini.
\newblock Uniform {H}\"older bounds for nonlinear {S}chr\"odinger systems with
  strong competition.
\newblock {\em Comm. Pure Appl. Math.}, 63(3):267--302, 2010.

\bibitem{ritorto}
Antonella Ritorto.
\newblock Optimal partition problems for the fractional {L}aplacian.
\newblock {\em Ann. Mat. Pura Appl. (4)}, 197(2):501--516, 2018.

\bibitem{sire.terracini.tortone}
Y.~{Sire}, S.~{Terracini}, and G.~{Tortone}.
\newblock {On the nodal set of solutions to degenerate or singular elliptic
  equations with an application to $s-$harmonic functions}.
\newblock {\em ArXiv e-prints}, August 2018.

\bibitem{zbMATH06870295}
Nicola {Soave}, Hugo {Tavares}, Susanna {Terracini}, and Alessandro {Zilio}.
\newblock {Variational problems with long-range interaction.}
\newblock {\em {Arch. Ration. Mech. Anal.}}, 228(3):743--772, 2018.

\bibitem{SoaveZilio_ARMA}
Nicola Soave and Alessandro Zilio.
\newblock Uniform bounds for strongly competing systems: the optimal
  {L}ipschitz case.
\newblock {\em Arch. Ration. Mech. Anal.}, 218(2):647--697, 2015.

\bibitem{MR2984134}
H.~Tavares and S.~Terracini.
\newblock Regularity of the nodal set of segregated critical configurations
  under a weak reflection law.
\newblock {\em Calc. Var. Partial Differential Equations}, 45(3-4):273--317,
  2012.

\bibitem{tvz2}
S.~Terracini, G.~Verzini, and A.~Zilio.
\newblock Uniform {H}\"older regularity with small exponent in
  competition-fractional diffusion systems.
\newblock {\em Discrete and Continuous Dynamical Systems- Series A},
  34(6):2669--2691, 2014.

\bibitem{tvz1}
S.~Terracini, G.~Verzini, and A.~Zilio.
\newblock Uniform {H}\"older bounds for strongly competing systems involving
  the square root of the laplacian.
\newblock {\em Journal of the European Mathematical Society},
  18(12):2865--2924, 2016.

\bibitem{tesi}
G.~Tortone.
\newblock {\em On the nodal set of solutions of degenerate - singular and
  nonlocal equations}.
\newblock 2018.
\newblock Thesis (Ph.D.)-- Politecnico di Torino and Universit\`a di Torino.

\bibitem{troianiello}
Giovanni~Maria Troianiello.
\newblock {\em Elliptic differential equations and obstacle problems}.
\newblock The University Series in Mathematics. Plenum Press, New York, 1987.

\bibitem{aletesi}
Alessandro Zilio.
\newblock {\em On monotonicity formulae, fractional operators and strong
  competition}.
\newblock PhD thesis, Politecnico di Milano, 2014.

\bibitem{ale}
Alessandro Zilio.
\newblock Optimal regularity results related to a partition problem involving
  the half-{L}aplacian.
\newblock In {\em New trends in shape optimization}, volume 166 of {\em
  Internat. Ser. Numer. Math.}, pages 301--314. Birkh\"{a}user/Springer, Cham,
  2015.

\end{thebibliography}
\end{document}